\documentclass[11pt]{article}
\usepackage[utf8]{inputenc}
\usepackage{amsfonts}
\usepackage{amsmath}
\usepackage{amsmath, amssymb}
\usepackage{amsthm}
\usepackage{mathtools}
\usepackage{enumerate}
\usepackage{comment}
\usepackage{xcolor}

\usepackage{chngcntr}
\usepackage{apptools}
\AtAppendix{\counterwithin{thm}{section}}

\theoremstyle{plain}
\newtheorem{thm}{Theorem}[section]
\newtheorem{prop}[thm]{Proposition}

\newtheorem{lem}[thm]{Lemma}
\theoremstyle{definition}
\newtheorem{defi}[thm]{Definition}
\newtheorem*{nota}{Notation}

\theoremstyle{remark}

\newcommand{\p}{\partial}

\newcommand{\E}{\mathcal{E}}

\newcommand{\Bo}[1]{\dot{B}^{#1}_{2,1}}
\newcommand{\Bi}[1]{\dot{B}^{#1}_{2,\infty}}

\newcommand{\dive}{\mathrm{div}}
\newcommand{\delj}{\dot{\Delta}_{j}}

\newcommand{\ep}{\epsilon}

\newcommand{\vertiii}[1]{{\left\vert\kern-0.25ex\left\vert\kern-0.25ex\left\vert #1 
    \right\vert\kern-0.25ex\right\vert\kern-0.25ex\right\vert}}

\begin{document}
\title{Stability analysis of time-periodic solutions to the Navier-Stokes-Fourier system in 3D whole space}
\author{Naoto Deguchi}
\author{Naoto Deguchi \thanks{\it{E-mail address} : \rm{deguchi.n.d813@m.isct.ac.jp}} \\ 
\footnotesize{\it{Institute of Science Tokyo, 2-12-1 Ookayama, Meguro-ku, Tokyo 152-8551, Japan}}
}
\date{}
\maketitle

\begin{abstract}
    This paper concerns the large-time behavior of perturbations around a time-periodic solution to the Navier-Stokes-Fourier system in the three-dimensional whole space. The time-periodic solution exists when a given external force is small enough. We derive the time-decay estimate of the perturbation under the assumption that an initial perturbation is sufficiently small. The time-space integral estimate for the linearized semigroup around the constant state in the Besov spaces is effectively applied in the proof.
\end{abstract}

\section{Introduction}

The motion of a compressible viscous and heat conducting fluid is described by the Navier-Stokes-Fourier system for the density $\rho$, the velocity $v=(v_1,v_2,v_3)$ and the temperature $\theta$:
\begin{equation} \label{eq}
    \left\{
    \begin{aligned}
        &\partial_t \rho + \dive(\rho v)=0\ \ \ \mathrm{in}\ \ \ \mathbb{R}\times\mathbb{R}^3,\\
        &\partial_t v + v\cdot \nabla v = \frac{\mu}{\rho}\Delta v + \frac{\mu+\mu'}{\rho}\nabla\dive\,v-\frac{\nabla P}{\rho} + f\ \ \ \mathrm{in}\ \ \ \mathbb{R}\times\mathbb{R}^3,\\
        &\partial_t \theta + v\cdot\nabla \theta + \frac{\theta \partial_\theta P}{c_V\rho}\dive\,v=\frac{\kappa}{c_V\rho } \Delta\theta + \frac{\Psi}{c_V\rho}\ \ \ \mathrm{in}\ \ \ \mathbb{R}\times\mathbb{R}^3.
    \end{aligned}
    \right.
\end{equation}
Here, $f=f(t,x)$ is the external body force, $P$ is the pressure, $\mu$ and $\mu'$ are the viscous coefficients, $\kappa$ is the heat conductivity coefficient, $c_V$ is the specific heat at constant volume and $\Psi=\Psi(v)$ is the dissipation function given by
\begin{align*}
    \Psi=\frac{\mu}{2}(\partial_i v^j+\partial_j v^i)^2 + \mu' (\partial_{i}v^i)^2.
\end{align*}
Throughout this paper, we assume that the coefficients $\mu$, $\mu'$, $\kappa$ and $c_V$ are constants which satisfy $\kappa,c_V>0$ and the thermodynamic restrictions 
\begin{align*}
    \mu>0,\ \ \ \frac{2}{3}\mu+\mu'\geq 0.
\end{align*}
The pressure $P$ is assumed to be a smooth function of $\rho$ and $\theta$ satisfying the stability conditions:
\begin{align*}
    \partial_{\rho} P>0\ \ \ \mathrm{and}\ \ \  \partial_{\theta}P>0.
\end{align*}
Let $f$ be a time-periodic force with time-period $T>0$, that is,
\begin{align*}
    f(t,x)=f(t+T,x)\ \ \ \mathrm{for\ any}\ \ \ t\in\mathbb{R},\ x\in\mathbb{R}^3.
\end{align*}
A flow generated by $f$ is then described as the solution to the following time-periodic problem:
\begin{equation} \label{teq}
    \left\{
    \begin{aligned}
        &\partial_t \rho_T + \dive(\rho_T v_T)=0\ \ \ \mathrm{in}\ \ \ \mathbb{R}\times\mathbb{R}^3,\\
        &\partial_t v_T + v_T\cdot \nabla v_T = \frac{\mu}{\rho_T}\Delta v_T + \frac{\mu+\mu'}{\rho_T}\nabla\dive\,v_T-\frac{\nabla P}{\rho_T} + f \ \ \ \mathrm{in}\ \ \ \mathbb{R}\times\mathbb{R}^3,\\
        &\partial_t \theta_T + v_T\cdot\nabla \theta_T + \frac{\theta_T \partial_\theta P}{c_V\rho_T}\dive\,v_T=\frac{\kappa}{c_V\rho_T } \Delta\theta_T + \frac{\Psi}{c_V\rho_T}\ \ \ \mathrm{in}\ \ \ \mathbb{R}\times\mathbb{R}^3,\\
        &(\rho_T,v_T,\theta_T)(t,x) = (\rho_T,v_T,\theta_T)(t+T,x)\ \ \ \mathrm{for}\ \ \ (t,x)\in\mathbb{R}\times\mathbb{R}^3.
    \end{aligned}
    \right.
\end{equation}
In this paper, we consider the time-periodic problem (\ref{teq}) around the constant state
\begin{align*}
    (\rho_\infty,0,\theta_\infty)\ \ \ \mathrm{with}\ \ \ \rho_\infty,\theta_\infty>0.
\end{align*}

The time-periodic problem has been extensively studied for the incompressible Navier-Stokes equation; we refer here to Maremonti \cite{MR1107016}, Kozono-Nakao \cite{MR1373173}, Yamazaki \cite{MR1777114}, Galdi-Silvestre \cite{MR2603768}, Kobayashi \cite{MR2589952} and Ku\v{c}era \cite{MR2610568} among many others. In particular, Yamazaki \cite{MR1777114} derived time-decay estimates for perturbations around a uniformly bounded in time solution, including a time-periodic solution, to the incompressible Navier-Stokes equation in unbounded domains such as the whole space, the half-space, or exterior domains of dimension $n \geq 3$ under the assumption of smallness for both the external force and the initial perturbation. For the isentropic compressible Navier-Stokes equation, the existence and stability of the time-periodic solution were studied by Valli \cite{MR753158} in a 3D bounded domain, and by Kagei-Tsuda \cite{MR3274764}, Tsuda \cite{MR3437858} in $\mathbb{R}^n$, $n\geq 3$. For the non-isentropic case, Ma, Ukai and Yang \cite{MR2595722} proved the existence and stability result for the compressible Navier-Stokes system in $\mathbb{R}^n$, $n\geq 5$ and Tsuda proved the existence and stability result for the compressible Navier-Stokes-Korteweg system in $\mathbb{R}^3$. We also mention that Feireisl et al. \cite{MR2917121} show the existence of at least one weak time periodic solution to the Navier-Stokes-Fourier system in a 3D bounded domain for large data. The time decay estimate of the perturbation around a time-periodic solution was derived by Ma, Ukai and Yang \cite{MR2595722} for the Navier-Stokes-Fourier system under the smallness assumptions. They prove the decay estimate
\begin{align} \label{udecay}
    &\|(\rho-\rho_T,v-v_T,\theta-\theta_T)(t)\|_{H^{N-1}}\nonumber \\
    &\hspace{50pt}\lesssim (1+t)^{-\frac{n}{4}} \|(\rho_0 - \rho_T(0), v_0 - v_T(0),\theta_0-\theta_T(0))\|_{L^1\cap H^{N-1}},
\end{align}
where $N\in\mathbb{Z}_{\geq n+2}$ and $H^{N-1}$ is the inhomogeneous Sobolev space defined in the end of this section. Here, $(\rho,v,\theta)$ is a global solution to \eqref{eq} in $\mathbb{R}^n$, $n\geq 5$ with initial data $(\rho_0,v_0,\theta_0)$ and $(\rho_T,v_T,\theta_T)$ is the corresponding time-periodic solution. The decay rate in \eqref{udecay} corresponds to that of the heat semigroup. However, the proof of decay estimate in \cite{MR2595722} relies on the dimensional assumption $n \geq 5$. In this paper, we aim to establish the decay estimates for perturbations around a time-periodic solution to the Navier-Stokes-Fourier system in the physically important three-dimensional case. 

 Let us state our main theorem, which derives the decay estimate for the perturbation.


\begin{thm} \label{timethm}
Let $T>0$ and $k\in\mathbb{Z}_{\geq 5}$. There exists a small $\delta>0$ such that if a time-periodic force $f$ with time period $T>0$ satisfies
\begin{align*}
    \|f\|_{C(\mathbb{R};\dot{B}^{-\frac{3}{2}}_{2,\infty}\cap\dot{H}^{k-1})}\leq \delta,
\end{align*}
then the following assertions hold. Here, the space $\dot{B}_{2,\infty}^{\frac{1}{2}}$ is the Besov space defined in Section $2$ below. 

\begin{enumerate}
    \item[$($\textup{i}$)$] \textup{(Existence of time-periodic solution)} There exists a unique time-periodic solution $(\rho_T,v_T, \theta_T)$ satisfying 
        \begin{align*}
            \|(\rho_T-\rho_\infty,v_T,\theta_T-\theta_\infty)\|_{C(\mathbb{R};\dot{B}^{\frac{1}{2}}_{2,\infty}\cap \dot{H}^k)} \lesssim \delta.
        \end{align*}
    \item[$($\textup{ii}$)$] \textup{(Stability)} If initial data $(\rho_0,v_0,\theta_0)$ satisfy
    \begin{align*}
        \|(\rho_0-\rho_\infty,v_0,\theta_0-\theta_\infty)\|_{\Bi{\frac{1}{2}}\cap \dot{H}^k}\leq \delta,
    \end{align*}
    then there exists a unique global solution $(\rho,v,\theta)$ of $(\ref{eq})$ in $(0,\infty)\times \mathbb{R}^3$ with initial data $(\rho_0,v_0,\theta_0)$ and has the property that
    \begin{align*}
        \|(\rho-\rho_\infty,v,\theta-\theta_\infty)\|_{C([0,\infty);\Bi{\frac{1}{2}}\cap\dot{H}^k)}\lesssim \delta.
    \end{align*}
    In addition, for any $\ep>0$, if $\delta>0$ is sufficiently small and
    \begin{align*}
        (\rho_0-\rho_T(0),v_0-v_T(0),\theta_0-\theta_T(0))\in L^p
    \end{align*}
    for some $1\leq p\leq 2$, then we have the time decay estimate:
    \begin{align} \label{maindecay}
        &\|(\rho-\rho_T,v-v_T,\theta-\theta_T)(t)\|_{\dot{H}^{s}} \nonumber\\
        &\hspace{10pt}\lesssim_s (1+t)^{-\frac{s}{2}-\frac{3}{2}\left(\frac{1}{p}-\frac{1}{2}\right)} \|(\rho_0-\rho_T(0),v_0-v_T(0),\theta_0-\theta_T(0))\|_{L^p\cap \dot{H}^k}
    \end{align}
    for any $-3/2+\ep\leq s\leq 3/2-\epsilon$ with $s/2+3/2(1/p-1/2)> 0$.
\end{enumerate}
\end{thm}

We note that the decay rate obtained in \eqref{maindecay} coincides with that of the heat semigroup. The employment of the Besov space is motivated by the need to control the slow spatial decay of the time-periodic solution. Indeed, since the diffusion terms 
\begin{align*}
    \mu\Delta v+(\mu+\mu')\nabla\dive v\ \ \ \mathrm{and}\ \ \ \frac{\kappa}{c_V}\Delta \theta
\end{align*}
are classified as elliptic operators, the spatial decay orders of the velocity $v_T$ and the temperature $\theta_T-\theta_\infty$ are expected at most $O(1/|x|)$ as $|x|\to\infty$, respectively, by comparing the spatial decay order of the Newtonian potential in $\mathbb{R}^3$. This slow decay suggests that the time-periodic solution does not belong to $L^2(\mathbb{R}^3)$. Previously, Ma, Ukai and Yang \cite{MR2595722} derived the decay estimate \eqref{udecay} in dimension $n\geq 5$ by combining the analysis using the linearized semigroup around the constant state $(\rho_\infty,0,\theta_\infty)$ and the energy estimates in $H^{N-1}$, $N\in\mathbb{Z}_{\geq n+2}$. In their work, the dimensional assumption $n\geq 5$ is used to control the $L^2$ norm of the time-periodic solution. To establish the decay estimate of the perturbation in dimension $n=3$, we develop the framework of perturbation analysis in the function spaces $\dot{B}^{1/2}_{2,\infty}\cap \dot{H}^k$, $k\geq \mathbb{Z}_{\geq 5}$, which spaces allow for slow decay, such as $1/|x|$. In addition, the spaces $\dot{B}^{1/2}_{2,\infty}\cap \dot{H}^k$, $k\geq \mathbb{Z}_{\geq 5}$ have good compatibility with the energy method developed by Matsumura and Nishida \cite{MR713680} for constructing strong solutions to the compressible Navier-Stokes-Fourier system, since the spaces are $L^2$-based space. Here, we mention that the space $\dot{B}^{1/2}_{2,1} \cap \dot{H}^3$ was used for the perturbation analysis of stationary solutions to the isentropic compressible Navier-Stokes equations in \cite{MR4803477}. 


The difficulty in proving Theorem \ref{timethm} arises from the convection terms in the Navier-Stokes-Fourier system. Since the spatial decay order of the time-periodic solution is expected to be at most ${O}(1/|x|)$, the convection term $v \cdot \nabla v$ appearing in the Navier-Stokes equations is expected to exhibit the same spatial decay order, ${O}(1/|x|^3)$, as the viscous term $\mu \Delta v$. Thus, we need to treat the convection terms carefully to prove the existence of a time-periodic solution and to derive the decay estimate for its perturbations even if the given external force is small enough. For this reason, Yamazaki and Tsuda transform the convection term $v\cdot\nabla v$ to the divergence forms $\dive(v\otimes v)$ and $\dive(\rho v\otimes v)$, respectively, by using the mass conservation low, and then establish a priori bound for the solutions to the Navier-Stokes equations.
However, in the Navier-Stokes-Fourier system case, the additional convection term $v\cdot\nabla \theta$ and the non-divergence form term $(\theta\partial_\theta P/c_V \rho)\dive v$ appear in the equation of the temperature in (\ref{eq}). In Section 3 below, to establish a priori estimate in the homogeneous Besov spaces, we shall introduce the modified energy functional
\begin{align} \label{modiene}
    &\mathcal{E}=\frac{1}{2}\left(1+\theta_\infty\frac{q_2}{p_2}\right)\rho v\cdot v + c_V\rho(\theta-\theta_\infty) \nonumber\\
    &\hspace{100pt}+\frac{p_1}{2\rho_\infty}\left(1+\theta_\infty\left(\frac{q_2}{p_2}-\frac{q_1}{p_1}\right)\right) (\rho-\rho_\infty)^2,
\end{align}
where $p_1=\p_\rho P(\rho_\infty,\theta_\infty)$, $p_2=\p_\theta P(\rho_\infty,\theta_\infty)$, $q_1=\p_\rho^2 P(\rho_\infty,\theta_\infty)$ and $q_2=\p_\rho \p_{\theta}P(\rho_\infty,\theta_\infty)$. By introducing the modified energy functional, let $\sigma=\rho-\rho_\infty$, $m=\rho v$ and $\eta=\theta-\theta_\infty$, then the equation (\ref{eq}) is transformed into the equation
\begin{equation}\label{eequ}
    \left\{
    \begin{aligned}
        &\partial_t \sigma + \dive\,m=0,\\
        &\partial_t m - \frac{\mu}{\rho_\infty}\Delta m -\frac{\mu+\mu'}{\rho_\infty}\nabla\dive\, m+ p_1\nabla \sigma +\frac{p_2}{c\rho_\infty}\nabla \E\\
        &\hspace{20pt}=-\frac{1}{\rho_\infty}\dive(m\otimes m)-\frac{p_2}{\rho_\infty}\nabla(\sigma \mathcal{E})-p_2\nabla\left(\eta-\frac{1}{c_V\rho_\infty}\mathcal{E}\right) +\rho f+\mathcal{H}_1,\\
        &\partial_t \E - \frac{\kappa}{c_V \rho_\infty}\Delta \E + \frac{ p_2\theta_\infty}{c_V\rho_\infty}\dive\,m=-c_V\dive(\eta m)\\
        &\hspace{10pt}- \left(1+\theta_\infty \frac{q_2}{p_2}\right)(p_1\dive(\sigma v)+ p_2 \dive(\eta v))+\left(1+\theta_\infty\frac{q_2}{p_2}\right)\rho_\infty f\cdot v + \mathcal{H}_2,
    \end{aligned}
    \right.
\end{equation}
where $\mathcal{H}_1$ and $\mathcal{H}_2$ are the higher order terms. Except for the external force terms and the higher-order terms, all nonlinear terms in (\ref{eequ}) are written in divergence forms or potential forms. This transform is used to establish a priori estimates for the initial value problem (\ref{eq}) in the proofs of Propositions \ref{lowd} and \ref{low} below. The proof of the time-decay estimate (\ref{maindecay}) is inspired by the strategy used in \cite{MR4803477}. The proof will be carried out by writing the perturbation $U=(\sqrt{p_1}\sigma,\ m,\ \sqrt{c_V/\theta_\infty}\E)$ as the Duhamel form:
\begin{align*}
    U(t)=e^{tA}U(0) + \int_0^t e^{\tau A}\Phi(t-\tau)d\tau
\end{align*}
where $t>0$, 
\begin{align*} 
    A=\left[\begin{matrix*}
        0 & -\sqrt{p_1} \dive & 0\\
        -\sqrt{p_1} \nabla & \frac{\mu}{\rho_\infty}\Delta + \frac{\mu+\mu'}{\rho_\infty}\nabla\dive  &- \frac{p_2}{\rho_\infty}\sqrt{\frac{\theta_\infty}{c_V}}\nabla\\
        0 & - \frac{p_2}{\rho_\infty}\sqrt{\frac{\theta_\infty}{c_V}}\dive & \frac{\kappa}{c\rho_\infty}\Delta
    \end{matrix*}\right],
\end{align*}
and $\Phi$ denotes the remainder terms that contain the linear terms such as  $\dive(v_T\otimes m)$. The key step of the proof is to derive the time-weighted estimate for the Duhamel integral:
\begin{align*}
    &\sup_{0\leq t\leq \widetilde{T}_1}(1+t)^{\frac{s-s_0}{2}}\left\|\int_0^t e^{\tau A}\Phi(t-\tau)d\tau\right\|_{\Bi{s}}\\
    &\hspace{50pt}\lesssim \sup_{s_1\leq \tilde{s}\leq\frac{3}{2}-\ep}\ \sup_{0\leq t\leq \widetilde{T}_1} (1+t)^{\frac{\tilde{s}-s_0}{2}}\|\Phi(t)\|_{\Bi{\tilde{s}-2}\cap\Bi{\frac{3}{2}}},
\end{align*}
where $\ep>0$, $-3/2\leq s_0\leq 1/2$, $s_1=\max\{s_0,0\}$, $\max\{s_0,-3/2+\ep\}\leq s\leq 3/2-\ep$. By virtue of rewriting the convection term as the divergence form, the bilinear estimate in Lemma \ref{bilinearlemma} below allow us to prove the estimates such as
\begin{align*}
    \|\dive(v_T\otimes m)\|_{\Bi{\tilde{s}-2}\cap\Bi{\frac{3}{2}}}\lesssim \|v_T\|_{\Bi{\frac{1}{2}}\cap\dot{H}^3} \|m\|_{\Bi{\tilde{s}}\cap\dot{H}^3},\quad s_1\leq \tilde{s}\leq 3/2-\ep,
\end{align*}
so that if the time-periodic solution is sufficiently small and the high-frequency norm $\dot{H}^3$ of the perturbation is well controlled, then the linear operator $\dive(v_T\otimes \cdot)$ can be treated as a simple perturbation of the operator $A$. The time-weighted estimate in the high-frequency norm is proved by using the energy method developed by Matsumura and Nishida \cite{MR713680} in Proposition \ref{highd} below.

The outline of this paper is as follows: In Section 2, we present some lemmas related to the homogeneous Besov space, which will be frequently used in the subsequent sections. Section 3 is devoted to the proof of Theorem \ref{timethm}. The proof of Theorem \ref{timethm} is established by combining the local existence theorem with a priori estimates for the solutions to the initial value problem of (\ref{eq}). A priori estimate and the local existence result for the initial value problem are proved in Subsection 3.1 and Subsection 3.2, respectively. We show the uniform in time and the time-weighted a priori estimate for the solution in Proposition \ref{aprioriesti} and Proposition \ref{diffdecayapri} below, respectively. A priori estimates are performed for the homogeneous Besov norms and the homogeneous Sobolev norms in a parallel way. The proof of the estimate in the homogeneous Besov spaces is carried out by using the modified energy functional introduced in (\ref{modiene}). Here, the time-space norm estimate in (\ref{semidecay3}) below for the linearized semigroup of (\ref{eequ}) around the constant state $(\rho_\infty,0,\theta_\infty)$ plays a crucial role in deriving the homogeneous Besov norm estimates. (Cf. \cite{MR1777114}, \cite{MR4803477}.) The proof of the estimates in the homogeneous Sobolev norms is established by the energy method. The proof of the local existence is carried out by constructing the approximating sequence of the solution. The approximating sequence is constructed as the solution of the transport equation for the density and as the solution of the quasilinear parabolic equation for the velocity and the temperature, since the density causes the quasilinear effect in the equations of the velocity and the temperature in (\ref{eq}). By virtue of employment in this construction, we obtain the local existence result without the smallness assumption for initial data, unlike the local existence result in \cite{MR4803477}. Finally, the proof of Theorem \ref{timethm} is given in Subsection 3.3.

\begin{nota}
    The notation $A\lesssim_{\alpha} B$ means that there exists a constant $C$ depending on $\alpha$ such that $A\leq CB$. The notation $A\sim_\alpha B$ means that $A\lesssim_\alpha B$ and $B\lesssim_\alpha A$. We denote a commutator by $[X,Y]= XY-YX$. We write $\mathcal{S}$ for the set of all Schwartz functions on $\mathbb{R}^{3}$, and we write $\mathcal{S}'$ for the set of all tempered distributions on $\mathbb{R}^{3}$. The notations $\hat{\cdot}$ and $\mathcal{F}$ stand for the Fourier transform
    \begin{align*}
        \hat{u}(\xi) = \mathcal{F}(u)(\xi)= \int_{\mathbb{R}^{3}} e^{-i x \cdot \xi} u(x)dx,
    \end{align*}
    and the notation $\mathcal{F}^{-1}$ denotes the inverse Fourier transform. We denote the $L^{2}(\mathbb{R}^3)$ inner product by 
    \begin{align*}
        \langle u,v\rangle = \int_{\mathbb{R}^3} u vdx.
    \end{align*}
    Let $Z$ be a Banach space, $1 \leq p \leq \infty$ and $I \subset \mathbb{R}$ be an interval. We denote by $L^p(I; Z)$ the space of all strongly measurable functions $f: I \to Z$ such that
    \begin{align*}
        \|f\|_{L^p(I;Z)}=\| \|f(t)\|_Z \|_{L^{p}_t(I)}<\infty.
    \end{align*}
    Similarly, $C(I; Z)$ denotes the space of continuous functions $f: I \to Z$ equipped with the norm
    \begin{align*}
        \|f\|_{C(I;Z)}=\sup_{t\in I} \|f(t)\|_Z<\infty.
    \end{align*}
    For $I = [0, T_0)$ with $0 < T_0 \leq \infty$, we write $L^p_{T_0}(Z)$ and $C_{T_0}(Z)$ instead of $L^p(I; Z)$ and $C(I; Z)$, respectively. In particular, if $T_0 = \infty$, we simply write $L^p(Z)$ and $C(Z)$ by omitting the subscript. For any $s\in\mathbb{R}$, $H^s=H^s(\mathbb{R}^3)$ denotes the inhomogeneous Sobolev space by
    \begin{align*}
        H^s=\{u\in \mathcal{S}'\mid (1+|\cdot|^2)^{\frac{s}{2}}\hat{u}\in L^2\}
    \end{align*}
    with norm
    \begin{align*}
        \|u\|_{H^s}=\|(1+|\cdot|^2)^{\frac{s}{2}}\hat{u}\|_{L^2}.
    \end{align*}
\end{nota}

\section{Preliminary} 

In this section, we present the definition and several lemmas for the homogeneous Besov space, which are frequently used in this paper. The following basic properties for the homogeneous Besov space are proved in \cite[Lemma 2.1, Proposition 2.22, Theorem 2.25 and Proposition 2.29]{MR2768550}. 
\begin{defi}
    Let $s\in\mathbb{R}$ and $1\leq r\leq \infty$. Choose $\phi\in C^\infty$ supported in the annulus $\mathcal{C}=\{\xi\in \mathbb{R}^3 \mid 3/4 \leq |\xi| \leq 8/3\}$ such that
    \begin{align*}
        \sum_{j\in\mathbb{Z}}\phi(2^{-j}\xi)=1\ \ \ \mathrm{for}\ \ \ \xi\neq 0.
    \end{align*}
    We define dyadic blocks $(\delj)_{j\in\mathbb{Z}}$ by the Fourier multiplier
    \begin{align*}
        \delj u \equiv \mathcal{F}^{-1} [\phi(2^{-j}\cdot) \hat{u}].
    \end{align*} 
    Then, the homogeneous Besov space $\dot{B}^{s}_{2,r}=\dot{B}^{s}_{2,r}(\mathbb{R}^3)$ is defined by
    \begin{align*}
        \dot{B}^{s}_{2,r}=\{u\in\mathcal{S}'/\mathcal{P}\mid \|u\|_{\dot{B}^{s}_{2,r}}<\infty\}
    \end{align*}
    with norm
    \begin{align*}
        \|u\|_{\dot{B}^{s}_{2,r}}:=\|\{2^{js}\|\delj u\|_{L^2}\}_{j\in\mathbb{Z}}\|_{\ell^{r}(\mathbb{Z})},
    \end{align*}
    where $\mathcal{P}$ is the polynomial ring over $\mathbb{R}$.
\end{defi}

The following lemma shows the basic properties of the Besov space.

\begin{lem} \label{fund2}
    Let $s,\tilde{s}\in \mathbb{R}$, $1\leq r,\tilde{r} \leq \infty$ and $u,v\in\mathcal{S}'$.
     \begin{enumerate}[\normalfont{(}i\normalfont{)}]
        \item \textup{(Derivative)} For any $k\in\mathbb{Z}_{\geq 0}$, it follows that $\|\nabla^{k}u\|_{\dot{B}_{2,r}^{s}}\sim\|u\|_{\dot{B}_{2,r}^{s+k}}$.
        
        \item \textup{(Duality)} Let $r'$ be a conjugate exponent of $r$. Then, we have the following duality estimate:
        \begin{align*}
            \langle u,v \rangle \lesssim \|u\|_{\dot{B}^{s}_{p,r}}\|u\|_{\dot{B}^{-s}_{2,r'}}\ \ \ \mathrm{and}\ \ \ \|u\|_{\dot{B}^{s}_{2,r}}\lesssim \sup_{\psi}\langle u,\psi\rangle,
        \end{align*}
        where the supremum is taken over the Schwartz functions $\psi$ with 
        \begin{align*}
            \|\psi\|_{\dot{B}^{-s}_{2,r'}} \leq 1
        \end{align*}
        and $0\notin \operatorname{supp}\mathcal{F}\psi$.
        \item \textup{(Interpolation)} Let $s_1 < s_2$ satisfy $s=(1-\theta)s_1 + \theta s_2$ for some $0<\theta<1$. Then, the interpolation inequality
        \begin{align} \label{interpolation}
            \|u\|_{\dot{B}^{s}_{2,r}} \lesssim_{\theta, s_1,s_2}  \|u\|^{1-\theta}_{\dot{B}^{s_1}_{2,\infty}}\|u\|^{\theta}_{\dot{B}^{s_2}_{2,\infty}} 
        \end{align}
        holds.
        
        \item \textup{(Fatou property)} Assume $s<3/2$ or $s=3/2$, $r=1$. If $\{u_n\}_{n\in\mathbb{Z}_{\geq 0}}$ is a bounded sequence in $\dot{B}_{2,r}^{s}\cap \dot{B}_{2,\tilde{r}}^{\tilde{s}}$, then there exists a subsequence of $\{u_n\}_n$ $($without relabeling$)$ and $u\in \dot{B}^{s}_{2,r}\cap \dot{B}_{2,\tilde{r}}^{\tilde{s}}$ such that 
        \begin{align*}
            \lim_{n\to \infty} u_{n} = u\ \ \mathrm{in}\ \ \mathcal{S}'\ \ \ \mathrm{and}\ \ \ \|u\|_{\dot{B}^{s}_{2,r}\cap \dot{B}_{2,\tilde{r}}^{\tilde{s}}} \lesssim \liminf_{n\to\infty} \|u_n\|_{\dot{B}^{s}_{2,r}\cap \dot{B}_{2,\tilde{r}}^{\tilde{s}}}.
        \end{align*}
    \end{enumerate}
\end{lem}

The following embedding property holds. (Cf. \cite[Proposition 2.20, Propositon 2.39 and Theorem 2.40]{MR2768550}.) 
\begin{lem} \label{em}
    For any $1\leq p\leq2$, the space $L^p$ is continuously embedded in the space $\Bi{-3(1/p-1/2)}$. For any $2\leq q\leq \infty$, the space $\Bo{3(1/2-1/q)}$ is continuously embedded in the space $L^q$.
\end{lem}
The following commutator estimate is the special case of \cite[Lemma 2.100 and Remark 2.102]{MR2768550}.
\begin{lem} \label{commu}
    Let $-3/2<s<5/2$, $1\leq r\leq \infty$. Then, we have
    \begin{align*}
        \left\|\left(2^{js}\|[\delj, v\partial_k]u\|_{L^2}\right)_{j\in\mathbb{Z}}\right\|_{\ell^{r}(\mathbb{Z})} \lesssim_{s}\|\nabla v\|_{\dot{B}^{\frac{3}{2}}_{2,1}} \|u\|_{\dot{B}^{s}_{2,r}}
    \end{align*}
    for any $1\leq k \leq 3$, where $u$ and $v$ are scalar functions.
\end{lem}

We state the bilinear estimate in the homogeneous Besov space. The following lemma is a direct consequence of \cite[Theorem 2.47 and Theorem 2.52]{MR2768550}.
\begin{lem} \label{bilinearlemma}
    Let $s_1,s_2\in\mathbb{R}$ satisfy $s_1,s_2<3/2$ and $s_1+s_2>0$. Let $1\leq r_1,r_2 \leq \infty$ satisfy $1/r_1+1/r_2=1/r$. Then, we have
    \begin{align*}
        \|uv\|_{\dot{B}_{2,r}^{s_1+s_2-\frac{3}{2}}} \lesssim_{s_1,s_2} \|u\|_{\dot{B}_{2,r_1}^{s_1}}\|v\|_{\dot{B}_{2,r_2}^{s_2}}.
    \end{align*}
    In the cases $s_1\leq 3/2$, $s_2 <3/2$ with $s_1+s_2\geq 0$, we have 
    \begin{align*}
        \|uv\|_{\dot{B}_{2,\infty}^{s_1+s_2-\frac{3}{2}}} \lesssim \|u\|_{\dot{B}_{2,1}^{s_1}}\|v\|_{\dot{B}_{2,\infty}^{s_2}}.
    \end{align*}
\end{lem}
We use the following lemmas regarding the composition of functions. (Cf. \cite[Theorem 2.61]{MR2768550}.)
\begin{lem} \label{comp1}
    Let $1\leq r\leq \infty$, $c>0$, $\Phi\in C^{\infty}(\mathbb{R})$ with $\Phi(0)=0$ and $\sigma\in\dot{B}^{s}_{2,r}\cap L^{\infty}$ with $s<2/3$, $1 \leq\infty$ or $s=3/2$, $r=1$ satisfies
    \begin{align*}
        \|\sigma\|_{L^\infty}\leq c.
    \end{align*}
    Then, we have
    \begin{align*}
        \|\Phi(\sigma)\|_{\dot{B}_{2,r}^s}\lesssim_{c,\Phi} \|\sigma\|_{\dot{B}_{2,r}^s}.
    \end{align*}
\end{lem}
By applying Lemma \ref{comp1} for the difference of composite functions
\begin{align*}
        \Phi(\sigma)-\Phi(\eta)=\int_0^1 (\Phi'(\eta +t(\sigma-\eta))-\Phi'(0))(\sigma-\eta)dt + \Phi'(0)(\sigma-\eta),
\end{align*}
we then obtain the following lemma. 
\begin{lem} \label{compos}
    Let $1\leq r\leq \infty$, $c>0$, $\Phi\in C^{\infty}(\mathbb{R})$ and $\sigma,\eta\in \dot{B}^{s}_{2,r}\cap \dot{B}^{3/2}_{2,1}$ with $-3/2\leq s<3/2$, $1\leq r\leq \infty$ or $s=3/2$, $r=1$ satisfy
    \begin{align*}
        \|(\sigma,\eta)\|_{L^\infty} \leq c.
    \end{align*}
    Then, we have
        \begin{align} \label{2compesti}
            \|\Phi(\sigma)-\Phi(\eta)\|_{\dot{B}^{s}_{2,r}} \lesssim_{c,\Phi} (1+\|(\sigma,\eta)\|_{\Bo{\frac{3}{2}}}) \|\sigma-\eta\|_{\dot{B}^{s}_{2,r}}.
        \end{align}
\end{lem}

\section{Existence and stability of time-periodic solution}

This section is devoted to the proof of Theorem \ref{timethm}. We shall investigate the initial value problem of (\ref{eq}):
\begin{equation} \label{inieq}
    \left\{
    \begin{aligned}
        &\partial_t \rho + \dive(\rho v)=0,\\
        &\partial_t v + v\cdot \nabla v = \frac{\mu}{\rho}\Delta v + \frac{\mu+\mu'}{\rho}\nabla\dive\,v-\frac{\nabla P}{\rho} + f,\\
        &\partial_t \theta + v\cdot\nabla \theta + \frac{\theta \partial_\theta P}{c_V\rho}\dive\,v=\frac{\kappa}{c_V\rho } \Delta\theta + \frac{\Psi}{c_V\rho},\\
        & (\rho,v,\theta)|_{t=0}=(\rho_0,v_0,\theta_0)
    \end{aligned}
    \right.
\end{equation}
in $(0,T_1)\times \mathbb{R}^3$ for given $T_1>0$.

\subsection{A priori estimate}
We start with a priori estimates for an initial value problem (\ref{inieq}).
\begin{prop} \label{aprioriesti}
    Let $k\in\mathbb{Z}_{\geq3}$, $0<T_1 <\infty$ and  
    \begin{align*}
        f\in C_{T_1}(\Bi{-\frac{3}{2}}\cap\dot{H}^{k-1}).
    \end{align*}
    Suppose that the pair of functions $(\rho,v,\theta)$ with
    \begin{align} \label{regassump}
        (\rho-\rho_\infty,v,\theta-\theta_\infty)\in C_{T_1}(\Bi{\frac{1}{2}}\cap\dot{H}^k)
    \end{align}
    is a solution of $(\ref{inieq})$ in $(0,{T_1})\times\mathbb{R}^3$. Then, there exists a constant $\delta_1>0$ such that if
    \begin{align} \label{apriassump}
        \|(\rho-\rho_\infty,v,\theta-\theta_\infty)\|_{C_{T_1}(\Bi{\frac{1}{2}}\cap\dot{H}^k)}+ \|f\|_{C_{T_1}(\Bi{-\frac{3}{2}}\cap\dot{H}^{k-1})}\leq \delta_1,
    \end{align}
    then we obtain
    \begin{align*}
        &\|(\rho-\rho_\infty,v,\theta-\theta_\infty)\|_{C_{T_1}(\Bi{\frac{1}{2}}\cap\dot{H}^{k})} \\
        &\hspace{30pt}\lesssim \|(\rho_{0}-\rho_\infty,v_{0},\theta_{0}-\theta_\infty)\|_{\Bi{\frac{1}{2}}\cap\dot{H}^k}+
         \|f\|_{C_{T_1}(\Bi{-\frac{3}{2}}\cap\dot{H}^{k-1})}.
    \end{align*}
\end{prop}
\begin{prop} \label{diffdecayapri}
    Let $k\in\mathbb{Z}_{\geq 5}$, $0<\ep<1/2$ and the pair of functions $(\rho_1,v_1,\theta_1)$, $(\rho_2,v_2,\theta_2)$ be solutions of $(\ref{inieq})$ in $(0,T_1)\times \mathbb{R}^3$ with initial data $(\rho_{0,1},v_{0,1},\theta_{0,1})$, $(\rho_{0,2},v_{0,2},\theta_{0,2})$, respectively. Suppose that the time-periodic force $f$ and the solutions $(\rho_1,v_1,\theta_1)$, $(\rho_2,v_2,\theta_2)$ satisfy $(\ref{regassump})$ and $(\ref{apriassump})$ with small $\delta_1>0$. Then, if the initial data satisfy
    \begin{align*}
        (\rho_{0,1}-\rho_{0,2},v_{0,1}-v_{0,2},\theta_{0,1}-\theta_{0,2})\in \Bi{s_0}
    \end{align*}
    for some $-3/2\leq s_0\leq 1/2$, then we have the time-weighted estimate
    \begin{align*}
        &\sup_{0\leq t\leq T_1}(1+t)^{\frac{s-s_0}{2}}\|(\rho_1-\rho_2,v_1-v_2,\theta_1-\theta_2)(t)\|_{\Bi{s}\cap \dot{H}^4}  \\
        &\hspace{100pt}\lesssim_s \|(\rho_{0,1}-\rho_{0,2},v_{0,1}-v_{0,2},\theta_{0,1}-\theta_{0,2})\|_{\Bi{s_0}\cap \dot{H}^{4}}
    \end{align*}
    for any $-3/2+\ep\leq s \leq 3/2-\ep$ with $s\geq s_0$.
\end{prop}
The proofs of Proposition \ref{aprioriesti} and Proposition \ref{diffdecayapri} below proceed in parallel. To derive a priori estimates in the homogeneous Besov spaces, we reformulate equation (\ref{eq}). 
Let $(\rho,v,\theta)$ be a solution of (\ref{eq}) with an external force $f$. Set
\begin{align*}
    \sigma=\rho-\rho_\infty,\ \ m=\rho v\ \ \ \mathrm{and}\ \ \ \eta=\theta-\theta_\infty.
\end{align*}
We introduce the modified energy functional $\mathcal{E}=\mathcal{E}(\sigma,v,\eta)$ by
\begin{align*}
    \mathcal{E}=\frac{1}{2}\left(1+\theta_\infty\frac{q_2}{p_2}\right)m\cdot v + c_V\rho\eta + \frac{p_1}{2\rho_\infty}\left(1+\theta_\infty\left(\frac{q_2}{p_2}-\frac{q_1}{p_1}\right)\right) \sigma^2,
\end{align*}
where $p_1=\p_\rho P(\rho_\infty,\theta_\infty)$, $p_2=\p_\theta P(\rho_\infty,\theta_\infty)$, $q_1=\p_\rho^2 P(\rho_\infty,\theta_\infty)$ and $q_2=\p_\rho \p_{\theta}P(\rho_\infty,\theta_\infty)$.
Since we have the identities
\begin{align*}
    \frac{1}{2} \p_t(\sigma^2) = -\rho_\infty \sigma \dive\, v - \sigma \dive(\sigma v),
\end{align*}
\begin{align*}
    &\frac{1}{2}\partial_{t}\left(m\cdot v\right) + \frac{1}{2}\dive \left((m\cdot v) v\right) \\
    &\hspace{50pt} = (\mu\Delta v+(\mu+\mu')\nabla\dive\,v)\cdot v-v
    \cdot\nabla P + f \cdot m,
\end{align*}
and
\begin{align*}
    \partial_t(\rho\eta) + \dive(\eta m) +\frac{\theta\partial_\theta P}{c_V}\dive v=\frac{\kappa}{c_V}\Delta\eta + \frac{\Psi}{c_V},
\end{align*}
the pair of functions $(\sigma,m,\E)$ satisfies the system of equations
\begin{equation} \label{energychap3}
    \left\{
    \begin{aligned}
        &\partial_t \sigma + \dive\,m=0,\\
        &\partial_t m - \frac{1}{\rho_\infty}\mathcal{A}m + p_1\nabla \sigma +\frac{p_2}{c_V \rho_\infty}\nabla \E = G+\rho f,\\
        &\partial_t \E - \frac{\kappa}{c_V\rho_\infty}\Delta \E + \frac{ p_2\theta_\infty}{c_V\rho_\infty}\dive\,m=H+\left(1+\theta_\infty\frac{q_2}{p_2}\right)f\cdot m,
    \end{aligned}
    \right.
\end{equation}
where $\mathcal{A}=\mu\Delta m - (\mu+\mu')\nabla\dive$, $p_1=\partial_\rho P(\rho_\infty,\theta_\infty)$, $p_2=\partial_\theta P(\rho_\infty,\theta_\infty)$, the functions $G=G(\sigma,v,\eta)$ and $H=H_1+H_2=H_1(\sigma,v,\eta)+H_2(\sigma,v,\eta)$ are given by
\begin{align*}
    G=-\dive\left(\frac{1}{\rho} m\otimes m\right) -p_2\nabla\mathcal{R}- \nabla Q_1 +\mathcal{A}\left(\left(\frac{1}{\rho}-\frac{1}{\rho_\infty}\right)m\right),
\end{align*}
\begin{align*}
    &H_1= -c_V\,\dive\left(\eta m\right) - \left(1+\theta_\infty \frac{q_2}{p_2}\right)(p_1\dive(\sigma v)+ p_2 \dive(\eta v)) \\
    &\hspace{70pt}-\frac{1}{2}\left(1+\theta_\infty\frac{q_2}{p_2}\right)\dive((m\cdot v)v)-\kappa\Delta\mathcal{R}\\
    &H_2=- \left(1+\theta_\infty \frac{q_2}{p_2}\right)v\cdot\nabla Q_1-Q_2 \dive\, v\\
    &\hspace{70pt}-\frac{p_1}{\rho_\infty}\left(1+\theta_\infty\left(\frac{q_2}{p_2}-\frac{q_1}{p_1}\right)\right) \sigma\dive(\sigma v) \\
    &\hspace{70pt}+\left(1+\theta_\infty\frac{q_2}{p_2}\right)(\mathcal{A}v)\cdot v+\Psi.
\end{align*}
Here, the functions $Q_1=Q_1(\sigma,\eta)$, $Q_2=Q_2(\sigma,\eta)$ and $\mathcal{R}=\mathcal{R}(\sigma,v,\eta)$ are given by
\begin{align*}
    Q_1 = P-P(\rho_\infty,\theta_\infty)-p_1\sigma-p_2\eta,
\end{align*}
\begin{align*}
    Q_2 = \theta \p_\theta P -\theta_\infty p_2 -\theta_\infty q_1 \sigma - (p_2 + \theta_\infty q_2)\eta
\end{align*}
and
\begin{align*}
    \mathcal{R}= \eta-\frac{1}{c_V \rho_\infty}\mathcal{E}.
\end{align*}

By rewriting the equations using the modified energy functional $\mathcal{E}$, the nonlinear terms $(G, H)$ in (\ref{energychap3}) are decomposed into $(G, H_1)$, in which all terms are in divergence or potential form, and $H_2$, which consists only of higher-order terms.

We shall rewrite the equation (\ref{energychap3}) by using the linearized semigroup around the constant state. Let
\begin{align*}
    U=U(\sigma,v,\eta)=\left(\sqrt{p_1}\sigma,\ m,\ \sqrt{\frac{c_V}{\theta_\infty}}\E\right)^{\mathsf{T}}.
\end{align*} 
Then, the Duhamel principle implies 
\begin{align*}
    U(t)=e^{tA}U_0 + \int_0^t e^{\tau A}\left[ \begin{matrix*} 0\\ G+\rho f \\ \sqrt{\frac{c}{\theta_\infty}}H +\left(1+\theta_\infty\frac{q_2}{p_2}\right)\rho f\cdot v \end{matrix*} \right](t-\tau) d\tau
\end{align*}
where $U_0 = U|_{t=0}$ and $e^{tA}$ is a Fourier multiplier with symbol $e^{t\hat{A}(\xi)}$, 
\begin{align} \label{ahatmat}
    \hat{A}(\xi)=\left[\begin{matrix*}
        0 & -i\sqrt{p_1} \xi^\mathsf{T} & 0\\
        -i\sqrt{p_1} \xi & -\frac{\mu}{\rho_\infty}|\xi|^2 I_3 + \frac{\mu+\mu'}{\rho_\infty}\xi\otimes \xi  &- i\frac{p_2}{\rho_\infty}\sqrt{\frac{\theta_\infty}{c_V}}\xi\\
        0 & - \frac{p_2}{\rho_\infty}\sqrt{\frac{\theta_\infty}{c_V}}\xi^{\mathsf{T}} & -\frac{\kappa}{c\rho_\infty}|\xi|^{2}
    \end{matrix*}\right],\ \ \ \xi\in\mathbb{R}^3.
\end{align}

The $5\times 5$ matrix $\hat{A}(\xi)$ has eigenvalues $\lambda_n(\xi)$, $n=1,\ldots,4$, which satisfy that there exist constants $r_0, \beta>0$, such that
\begin{align} \label{rambda1}
    \operatorname{Re}\lambda_n(\xi) \sim -|\xi|^{2}\ \ \ \mathrm{for}\ \ \ |\xi|\leq r_0
\end{align}
and
\begin{align} \label{rambda2}
    \operatorname{Re}\lambda_n(\xi) \leq -\beta\ \ \ \mathrm{for}\ \ \ |\xi|>r_0
\end{align}
for $n=1,\ldots,4$. Furthermore, the spectral resolutions
\begin{align} \label{projec}
    e^{t\hat{A}(\xi)}=\sum_{n=1}^4 e^{t\lambda_{n}(\xi)}P_{n}(\xi),\ \ \  \xi\in\mathbb{R}^3,
\end{align}
hold, where $P_n(\xi)$ is the corresponding projection matrix.
The proofs of $(\ref{rambda1})$, $(\ref{rambda2})$ and $(\ref{projec})$ are found in \cite[Lemma 3.1]{MR555060}, \cite[Lemma 4.2]{MR553964}. Define the low-frequency cut-off operator $\dot{S}_{j}$, $j\in\mathbb{Z}$ by 
\begin{align*}
    \dot{S}_{j}u=\sum_{j'<j} \dot{\Delta}_{j'} u,\ \ \ u\in\mathcal{S}'.
\end{align*}
Then, the spectral properties $(\ref{rambda1})$ and $(\ref{rambda2})$ give the following estimate for the semigroup $e^{tA}$. Here, we note that Danchin \cite[Proposition 2.2]{MR1886005} has already shown the following estimates for the isentropic compressible Navier-Stokes equations.


\begin{lem} \label{decayinfty}
    Let $s,s_0\in\mathbb{R}$ satisfy $s_0\leq s$. Then, there exists $j_0\in\mathbb{Z}$ such that
    \begin{align} \label{semidecay}
        \|\dot{S}_{j_0}e^{tA}V_0\|_{\dot{B}^{s}_{2,r}}\lesssim (1+t)^{-\frac{s-s_0}{2}}\|V_0\|_{\dot{B}^{s_0}_{2,r}},\ \ \ t\geq0,
    \end{align}
    for any $1\leq r\leq \infty$, $V_0\in \Bi{s_0}$ and 
    \begin{align} \label{semidecay2}
        \|(1-\dot{S}_{j_0})e^{tA}V_0\|_{\dot{B}^{s}_{2,r}}\lesssim e^{-ct}\|V_0\|_{\dot{B}^{s_0}_{2,r}},\ \ \ t\geq0,
    \end{align}
    for any $1\leq r\leq \infty$, $V_0\in \Bi{s}$, where $c>0$ is a constant.
    Furthermore, for any $1\leq r\leq \infty$, we have 
    \begin{align} \label{semidecay3}
        \|\dot{S}_{j_0}e^{tA}V_0\|_{L^r(\dot{B}^{s}_{2,r})} \lesssim \|V_0\|_{\dot{B}^{s-\frac{2}{r}}_{2,r}},\ \ \ V_0\in \dot{B}^{s-\frac{2}{r}}_{2,r}.
    \end{align}
    The estimates $(\ref{semidecay})$, $(\ref{semidecay2})$ and $(\ref{semidecay3})$ remain valid when $A$ is replaced by its adjoint $A^*$.
\end{lem}

\begin{proof}
    From (\ref{rambda1}) and (\ref{rambda2}), it follows that there exist constants $c_1,c_2>0$ and $j_0\in\mathbb{Z}$ such that
    \begin{align*}
        |\phi(2^{-j}\xi)e^{t\lambda_n(\xi)}| \lesssim 
        \begin{cases}
            e^{-c_1t|\xi|^2} & \text{if $j< j_0$,} \\
            e^{-c_2 t}       & \text{if $j\geq j_0$.} 
        \end{cases}
    \end{align*}
    for any $n=1,\ldots,4$ and $\xi\in\mathbb{R}^3$. Then, the spectral resolutions (\ref{projec}) yield
    \begin{align} \label{2je}
        2^{sj}\|\delj \dot{S}_{j_0}e^{tA}V_0\|_{L^{2}}\lesssim_{r_0} 2^{sj}e^{-\tilde{c}_12^{2j}t} \|\delj V_0\|_{L^2}
    \end{align}
    and
    \begin{align} \label{2je2}
        2^{sj}\|\delj (1-\dot{S}_{j_0})e^{tA}V_0\|_{L^{2}}\lesssim_{r_0} e^{-\tilde{c}_2 t} 2^{sj}\|\delj V_0\|_{L^2},
    \end{align}
    where $\tilde{c}_1,\tilde{c}_2>0$ are constants. By taking the $\ell^{r}(\mathbb{Z})$ norm and the $L^r(\ell^{r}(\mathbb{Z}))$ norm, respectively, of both sides of (\ref{2je}), we obtain the estimates (\ref{semidecay}) and (\ref{semidecay3}), respectively. The estimate (\ref{semidecay2}) is obtained by taking the $\ell^{r}(\mathbb{Z})$ norm of both sides of (\ref{2je2}).
\end{proof}

The following lemma will be used to derive the time-weighted estimate. The proof of Lemma \ref{phidecay} below is the same as in \cite[Proposition 5.2]{MR4803477}.
\begin{lem} \label{phidecay}
    Let $\widetilde{T}_1>0$, $-3/2\leq s_0\leq 1/2$ and $\Phi\in L^\infty_{T_2}(\Bi{s_1-2}\cap\Bi{3/2})$ with  $s_{1}=\max\{s_0,0\}$. For any $0<\ep<1/2$, we set
    \begin{align*}
        D_\ep= \sup_{s_1\leq s\leq\frac{3}{2}-\ep}\ \sup_{0\leq t\leq \widetilde{T}_1} (1+t)^{\frac{s-s_0}{2}}\|\Phi(t)\|_{\Bi{s-2}\cap\Bi{\frac{3}{2}}},
    \end{align*}
    then we have 
    \begin{align} \label{inhomesti}
        \sup_{0\leq t\leq \widetilde{T}_1}(1+t)^{\frac{s-s_0}{2}}\left\|\int_0^t e^{\tau A}\Phi(t-\tau)d\tau\right\|_{\Bi{s}}\lesssim_{\ep} D_\ep
    \end{align}
    for any $s\in\mathbb{R}$ with $\max\{s_0,-3/2+\ep\}\leq s\leq 3/2-\ep$. 
\end{lem}
\begin{proof}
    We first show the estimate (\ref{inhomesti}) with $s=3/2-\ep$. Let
    \begin{align*}
        \alpha_0 = 
        \begin{cases}
            3 & \text{if $-3/2\leq s_0\leq 0$,} \\
            5/2 & \text{if $0 <s_0\leq 1/2$.} 
        \end{cases}
    \end{align*}
    For any test function $\psi\in\mathcal{S}$, by Lemma \ref{fund2} (ii),
    \begin{align*}
        \left\langle\dot{S}_{j_0}\int_0^{t} e^{\tau A}\Phi(t-\tau)d\tau, \psi\right\rangle &= \int_0^{t} \langle \Phi(t-\tau),\dot{S}_{j_0}e^{\tau A^*}\psi\rangle d\tau\\
        &\lesssim \int_0^{\frac{t}{2}} \|\Phi(t-\tau)\|_{\Bi{-\frac{1}{2}-\ep}}\|\dot{S}_{j_0}e^{\tau A^*}\psi\|_{\Bo{\frac{1}{2}+\ep}} d\tau\\
        &\hspace{5pt}+\int_{\frac{t}{2}}^t \|\Phi(t-\tau)\|_{\Bi{\frac
        {3}{2}-\ep-\alpha_0}}\|\dot{S}_{j_0}e^{\tau A^*}\psi\|_{\Bo{\alpha_0-\frac{3}{2}+\ep}} d\tau.
    \end{align*}
    By Lemma \ref{decayinfty}, we have
    \begin{align*}
        \|\dot{S}_{j_0}e^{\tau A^*}\psi\|_{\Bo{\alpha_0-\frac{3}{2}+\ep}} &\lesssim (1+\tau)^{-\frac{\alpha_0}{2}}\|\psi\|_{\Bo{-\frac{3}{2}+\ep}}\\
        &\lesssim (1+t)^{-\frac{\alpha_0}{2}}\|\psi\|_{\Bo{-\frac{3}{2}+\ep}} ,\ \ \ \frac{t}{2}\leq \tau\leq t,
    \end{align*}
    and
    \begin{align*}
        \int_0^{\frac{t}{2}} \|\dot{S}_{j_0}e^{\tau A^*}\psi\|_{\Bo{\frac{1}{2}+\ep}} d\tau\lesssim \|\psi\|_{\Bo{-\frac{3}{2}+\ep}},
    \end{align*}
    so that we obtain
    \begin{align*}
        \left\langle\dot{S}_{j_0}\int_0^{t} e^{\tau A}\Phi(t-\tau)d\tau, \psi\right\rangle \lesssim (1+t)^{-\frac{3}{4}+\frac{s_0+\ep}{2}}D_\ep.
    \end{align*}
    Lemma \ref{fund2} (ii) now shows that
    \begin{align*}
        \left\|\dot{S}_{j_0}\int_0^{t} e^{\tau A}\Phi(t-\tau)d\tau\right\|_{\Bi{\frac{3}{2}-\ep}}\lesssim (1+t)^{-\frac{3}{4}+\frac{s_0+\ep}{2}}D_\ep.
    \end{align*}
    By Lemma \ref{decayinfty}, 
    \begin{align*}
        \left\|(1-\dot{S}_{j_0})\int_0^{t} e^{\tau A}\Phi(t-\tau)d\tau\right\|_{\Bi{\frac{3}{2}-\ep}}&\lesssim \int_{0}^t e^{-ct} \|\Phi(t-\tau)\|_{\Bi{\frac{3}{2}-\ep}} d\tau\\
        &\lesssim (1+t)^{-\frac{3}{4}+\frac{s_0+\ep}{2}}D_\ep.
    \end{align*}
    Then, we obtain the estimate (\ref{inhomesti}) with $s=3/2-\ep$. Next, we show the estimate (\ref{inhomesti}) with $s=\max\{s_0,-3/2+\ep\}$. If $s<1/2$, then by Lemma \ref{decayinfty}, we have
    \begin{align*}
        \left\|\int_0^{t} e^{\tau A}\Phi(t-\tau)d\tau\right\|_{\Bi{s}} &\lesssim \int_{0}^t (1+\tau)^{-\frac{3}{4}-\frac{s_0}{2}}\|\Phi(t-\tau)\|_{\Bi{-\frac{3}{2}}\cap\Bi{s}}d\tau \\
        &\lesssim D_\ep \int_0^t (1+\tau)^{-\frac{3}{4}-\frac{s}{2}}(1+t-\tau)^{-\frac{1}{4}+\frac{s}{2}}d\tau\\
        &\lesssim D_\ep.
    \end{align*}
    If $s=1/2$, then Lemma \ref{decayinfty} implies that
    \begin{align*}
        \left\|\int_0^{t} e^{\tau A}\Phi(t-\tau)d\tau\right\|_{\Bi{\frac{1}{2}}}\lesssim D_\ep.
    \end{align*}
    By interpolating the estimates in (\ref{inhomesti}) with $s=\max\{s_0,-3/2+\ep\}$ and $s=3/2-\ep$, we obtain the estimate (\ref{inhomesti}) for any $\max\{s_0,-3/2+\ep\}\leq s \leq 3/2-\ep$.
\end{proof}
The following proposition shows the time-weighted estimate in $\Bi{s}$ for the differences
\begin{align*}
    \tilde{\sigma}:=\rho_1-\rho_2,\ \tilde{v}:=v_1-v_2,\ \tilde{\eta}:=\eta_1-\eta_2.
\end{align*}
\begin{prop} \label{lowd}
    Under the assumption of Proposition $\ref{diffdecayapri}$ with small $\delta_1>0$, for any $-3/2+\ep \leq s\leq 3/2-\ep$ and $-3/2\leq s_0\leq 1/2$ with $s_0\leq s$, we have
    \begin{align*} 
        \sup_{0\leq t\leq T_1}(1+t)^{\frac{s-s_0}{2}}\|(\tilde{\sigma},\tilde{v},\tilde{\eta})(t)\|_{\Bi{s}}\lesssim_{\ep} \delta_1 \mathcal{D}_{\ep}(T_1)+ \|(\tilde{\sigma}_0,\tilde{v}_0,\tilde{\eta}_0)\|_{\Bi{s_0}\cap\Bi{\frac{3}{2}-\ep}},
    \end{align*}
     where $\tilde{\sigma}_0=\rho_{0,1}-\rho_{0,2}$, $\tilde{v}_0=v_{0,1}-v_{0,2}$, $\tilde{\eta}_0=\theta_{0,1}-\theta_{0,2}$ and
     \begin{align*}
        \mathcal{D}_{\ep}(T_1)=\sup_{s_1\leq \tilde{s}\leq \frac{3}{2}-\ep} \sup_{0\leq t\leq T_1} (1+t)^{\frac{\tilde{s}-s_0}{2}}\|(\tilde{\sigma},\tilde{v},\tilde{\eta})\|_{\Bi{\tilde{s}}\cap\dot{H}^4}, \ \ s_1=\max\{0,s_0\}.
    \end{align*}
\end{prop}
\begin{proof}
    By Lemma \ref{bilinearlemma} and the assumption (\ref{apriassump}), if $\delta_1>0$ is small enough, then
    \begin{align*}
        \|(\tilde{\sigma},\tilde{v},\tilde{\eta})\|_{\Bi{s}} \sim_{\delta_1,\ep} \|\tilde{U}\|_{\Bi{s}}
    \end{align*}
    for any $-3/2+\ep\leq s\leq3/2-\ep$, where $\tilde{U}:=U(\sigma_1,v_1,\eta_1)-U(\sigma_2,v_2,\eta_2)$.
    Thus, by Lemma \ref{phidecay}, it suffices to show that
    \begin{align} \label{tilgtilh}
        \|(\tilde{G},\tilde{H}_1) \|_{\Bi{s-2}\cap\Bi{\frac{3}{2}}} \lesssim \delta_1  \|(\tilde{\sigma},\tilde{v},\tilde{\eta})\|_{\Bi{s}\cap\dot{H}^4}
    \end{align}
    for any $s_1\leq s\leq 3/2-\ep$, where $\tilde{G}=G(\sigma_1,v_1,\eta_1)-G(\sigma_2,v_2,\eta_2)$, $\tilde{H}_1=H_1(\sigma_1,v_1,\eta_1)-H_1(\sigma_2,v_2,\eta_2)$, and
    \begin{align} \label{gammadecay}
        \left\|\int_0^t e^{\tau A}\tilde{\Gamma}(t-\tau) d\tau\right\|_{\Bi{s}}\lesssim \delta_1(1+t)^{-\frac{s-s_0}{2}}\mathcal{D}_{\ep}(T_1)
    \end{align}
    for any $ s_1\leq s\leq 3/2-\ep$, where $\tilde{\Gamma}=\Gamma(\sigma_1,v_1,\eta_1)-\Gamma(\sigma_2,v_2,\eta_2)$ with
    \begin{align*}
        \Gamma(\sigma,v,\eta):=\left[ \begin{matrix*} 0\\ \rho f \\ \sqrt{\frac{c}{\theta_\infty}}H_2(\sigma,v,\eta)+\left(1+\theta_\infty\frac{q_2}{p_2}\right)\rho f\cdot v \end{matrix*} \right].
    \end{align*}
    We first show the estimate (\ref{gammadecay}). 
    Applying Taylor's theorem for $Q_1=Q_1(\sigma,\eta)$, $Q_2=Q_2(\sigma,\eta)$, respectively, there are smooth functions $R_1=R_1(\sigma,\eta)$, $R_2=R_2(\sigma,\eta)$ with $R_1(0,0)=R_2(0,0)=0$ such that
    \begin{align*}
        Q_1=\frac{1}{2}[\sigma,\eta](D^2P(\rho_\infty,\theta_\infty)+R_1)[\sigma,\eta]^{\mathsf{T}} 
    \end{align*}
    and
    \begin{align*}
        Q_2=\frac{1}{2}[\sigma,\eta](D^2\partial_\theta P(\rho_\infty,\theta_\infty)+R_2)[\sigma,\eta]^{\mathsf{T}},
    \end{align*}
    where $D^2P$ and $D^2\partial_\theta P$ are the Hesse matrices.
    Since terms $v\cdot\nabla Q_1(\sigma,\eta)$ and $Q_2(\sigma,\eta)\dive\, v$ consist of product of at least three functions, Lemma \ref{bilinearlemma} shows that
    \begin{align*}
        \|v_1\cdot\nabla Q_1(\sigma_1,\eta_1)-v_2\cdot\nabla Q_1(\sigma_2,\eta_2)\|_{\Bi{-\frac{3}{2}}\cap\Bi{\frac{3}{2}}} \lesssim_{\delta_1} \delta_1^2 \|(\tilde{\sigma},\tilde{\eta})\|_{\Bi{\frac{3}{2}-\ep}\cap\dot{H}^3}
    \end{align*}
    and
    \begin{align*}
        \|Q_2(\sigma_1,\eta_1)\dive\, v_1-Q_2(\sigma_2,\eta_2)\dive\, v_2\|_{\Bi{-\frac{3}{2}}\cap\Bi{\frac{3}{2}}} &\lesssim_{\delta_1} \delta_1^2 \|(\tilde{\sigma},\tilde{v},\tilde{\eta})\|_{\Bi{\frac{3}{2}-\ep}\cap\dot{H}^3}.
    \end{align*}
    Thus, Lemma \ref{bilinearlemma} shows
    \begin{align*}
        \|H_2(\sigma_1,v_1,\eta_1)-H_2(\sigma_2,v_2,\eta_2)\|_{\Bi{-\frac{3}{2}}\cap\Bi{\frac{3}{2}}}\lesssim \delta_1 \|(\tilde{\sigma},\tilde{v},\tilde{\eta})\|_{\Bi{\frac{3}{2}-\ep}\cap\dot{H}^3}.
    \end{align*}
    By Lemma \ref{bilinearlemma} and Lemma \ref{decayinfty}, we have
    \begin{align*}
        &\left\|\int_0^te^{tA}\tilde{\Gamma}(t-\tau)d\tau\right\|_{\Bi{s}}\lesssim \int^t_0 (1+\tau)^{-\frac{3}{4}-\frac{s}{2}}\|\tilde{\Gamma}(t-\tau)\|_{\Bi{-\frac{2}{3}}\cap\Bi{s}} d\tau\\
        &\hspace{50pt}\lesssim \int_0^t (1+\tau)^{-\frac{3}{4}-\frac{s}{2}}\|\tilde{\Gamma}(t-\tau)\|_{\Bi{-\frac{2}{3}}\cap\Bi{s}} d\tau\\
        &\hspace{50pt}\lesssim (\|f\|_{\Bo{-\frac{3}{2}+\ep}}+\delta_1^2)\int_0^t (1+\tau)^{-\frac{3}{4}-\frac{s}{2}}\|(\tilde{\sigma},\tilde{v},\tilde{\eta})\|_{\Bi{\frac{3}{2}-\ep}\cap\dot{H}^3} d\tau\\
        &\hspace{50pt}\lesssim \delta_1\mathcal{D}_{\ep}(T_1)(1+t)^{-\frac{s-s_0}{2}}.
    \end{align*}
    We proceed to show that the estimate in (\ref{tilgtilh}). By Lemma \ref{bilinearlemma}, for any $s_1\leq s\leq 3/2-\ep$, we have
    \begin{align*}
        &\left\|\dive\left(\rho_1 v_1\otimes v_1\right)-\dive\left(\rho_2 v_2\otimes v_2\right)\right\|_{\Bi{s-2}\cap\Bi{\frac{3}{2}}} \\
        &\hspace{30pt}\lesssim \left\|\rho_1 v_1\otimes v_1-\rho_2 v_2\otimes v_2\right\|_{\Bi{s-1}\cap\Bi{\frac{5}{2}}}\\
        &\hspace{30pt}\lesssim \delta_1 \|(\tilde{\sigma},\tilde{v})\|_{\Bi{s}\cap\dot{H}^3},
    \end{align*}
    \begin{align*}
        \left\|\nabla\mathcal{R}(\sigma_1,v_1,\eta_1)-\nabla \mathcal{R}(\sigma_2,v_2,\eta_2)\right\|_{\Bi{s-2}\cap\Bi{\frac{3}{2}}} \lesssim  \delta_1 \|(\tilde{\sigma},\tilde{v},\tilde{\eta})\|_{\Bi{s}\cap\dot{H}^3},
    \end{align*}
    \begin{align*}
        &\left\|\Delta\mathcal{R}(\sigma_1,v_1,\eta_1)-\Delta \mathcal{R}(\sigma_2,v_2,\eta_2)\right\|_{\Bi{s-2}\cap \Bi{\frac{3}{2}}} \\
        &\hspace{30pt}\lesssim  \left\|\mathcal{R}(\sigma_1,v_1,\eta_1)- \mathcal{R}(\sigma_2,v_2,\eta_2)\right\|_{\Bi{s}\cap\Bi{\frac{7}{2}}}\\
        &\hspace{30pt}\lesssim \delta_1 \|(\tilde{\sigma},\tilde{v},\tilde{\eta})\|_{\Bi{s}\cap\dot{H}^4},
    \end{align*}
    and
    \begin{align*}
        &\|\dive((\rho_1v_1\cdot v_1)v_1)-\dive((\rho_2v_2\cdot v_2)v_2)\|_{\Bi{s-2}\cap\Bi{\frac{3}{2}}}\\
        &\hspace{30pt}\lesssim \|(\rho_1v_1\cdot v_1)v_1-(\rho_2v_2\cdot v_2)v_2\|_{\Bi{s-1}\cap\Bi{\frac{5}{2}}}\\
        &\hspace{30pt}\lesssim \delta_1 \|(\tilde{\sigma},\tilde{v})\|_{\Bi{s}\cap\dot{H}^3}
    \end{align*}
    The estimates for the remaining terms in $\tilde{G}$ and $\tilde{H}_1$ follow similarly.
\end{proof}
By repeating the proof of Proposition \ref{lowd} for $(\sigma,v,\eta)$ in the case $s=s_0=1/2$, we obtain the following proposition.
\begin{prop} \label{low}
    Under the assumption of Proposition $\ref{aprioriesti}$ with small $\delta_1>0$, we have
    \begin{align*} 
        &\|(\sigma,v,\eta)\|_{C_{T_1}(\Bi{\frac{1}{2}})} \lesssim \delta_1 \|(\sigma,v,\eta)\|_{C_{T_1}(\Bi{\frac{1}{2}}\cap\dot{H}^3)}\nonumber \\
        &\hspace{60pt}+ \|(\sigma_0,v_0,\eta_0)\|_{\Bi{\frac{1}{2}}}+\|f\|_{C_{T_1}(\Bi{-\frac{3}{2}}\cap\Bi{-\frac{1}{2}})},
    \end{align*}
    where $\sigma_0=\rho_0-\rho_\infty$ and $\eta_0=\theta_0-\theta_\infty$.
\end{prop}

The estimates in the homogeneous Sobolev spaces $\dot{H}^k$ are shown using the energy method for a solution to the initial value problem (\ref{inieq}).

\begin{prop} \label{high}
    Under the assumption of Proposition $\ref{aprioriesti}$ with small $\delta_1>0$, we have
    \begin{align}  \label{highesti}
        &\|(\sigma,v,\eta)\|_{C_{T_1}(\dot{H}^1\cap\dot{H}^k)} \lesssim \|(\sigma,v,\eta)\|_{C_{T_1}(\Bi{\frac{1}{2}})}\nonumber \\
        &\hspace{60pt}+ \|(\sigma_0,v_0,\eta_0)\|_{\dot{H}^1\cap\dot{H}^{k}}+\|f\|_{C_{T_1}({H}^{k-1})}.
    \end{align}
\end{prop}
\begin{proof}
    Since $(\sigma,v,\eta)$ satisfies (\ref{eq}), we have the energy identities:
    \begin{align} \label{energyid}
        &\frac{1}{2} \p_t \left\|\nabla^\ell\left(\frac{p_1}{\rho_\infty^2}\sigma,\, v, \,\frac{c}{\theta_\infty}\eta\right)\right\|_{L^2}^2 \nonumber\\
        &+ \frac{\mu}{\rho_\infty}\|\nabla^{\ell+1} v\|_{L^2}^2 + \frac{\mu+\mu'}{\rho_\infty}\|\nabla^{\ell}\dive\,v\|_{L^{2}}^2 + \frac{\kappa}{\rho_\infty \theta_\infty}\|\nabla^{\ell+1}\eta\|_{L^2}^2 \nonumber\\
        &= \frac{p_1}{\rho_\infty^2}\langle \nabla^{\ell}\dive(\sigma v), \nabla^{\ell} \sigma\rangle + \langle \nabla^\ell R_1,\nabla^\ell v\rangle + \left\langle\frac{1}{\theta} \nabla^\ell R_2, \nabla^{\ell}\eta\right\rangle
    \end{align}
    for $\ell=1,\ldots,k$, and
    \begin{align} \label{energyidsigma}
        &\partial_t \langle \nabla^{\ell-1} v, \nabla^{\ell-1} \nabla\sigma \rangle +  \frac{P'(\rho_\infty)}{\rho_\infty}\|\nabla^\ell\sigma\|_{L^2}^2 \nonumber\\
        &= \rho_\infty\|\nabla^{\ell-1}\dive\,v\|_{L^2}^2 -\langle\nabla^{\ell-1}v,\nabla^{\ell-1}\nabla \dive(\sigma v)\rangle  \nonumber\\
        &\hspace{70pt}+\left\langle \nabla^{\ell-1}\left(\frac{1}{\rho_\infty}\mathcal{A}v+R_1\right),\nabla^{\ell-1}\nabla\sigma\right\rangle,
    \end{align}
    for any $\ell=2,\ldots,k$, where the functions $R_1=R_1(\sigma,v,\eta)$, $R_2=R_2(\sigma,v,\eta)$ and $R_3=R_3(\sigma,v,\eta)$ are given by
    \begin{align*}
        &R_1 = -v\cdot \nabla v +\left(\frac{1}{\rho}-\frac{1}{\rho_\infty}\right)\mathcal{A}v\\
        & \hspace{50pt}- \left(\frac{1}{\rho}-\frac{1}{\rho_\infty}\right)(p_1\nabla\sigma+p_2\nabla\eta) - \frac{\nabla Q_1}{\rho}+f
    \end{align*}
    and
    \begin{align*}
        R_2= -c v\cdot\nabla\eta -\left(\frac{\theta\p_\theta P}{\rho}-\frac{\theta_\infty p_2}{\rho_\infty} \right)\dive\,v +\kappa\left(\frac{1}{\rho}-\frac{1}{\rho_\infty}\right)\Delta \theta + \frac{\Psi}{\rho}.
    \end{align*}                                                    
    For any multi-index $\alpha\in\mathbb{Z}_{\geq 0}^3 $ with $|\alpha|=k$, 
    \begin{align} \label{alphaiden}
        \langle \p^\alpha_x \dive(\sigma v),\p_x^{\alpha}\sigma \rangle = -\frac{1}{2} \langle \dive\,v \p^\alpha_x \sigma,\p^\alpha_x\sigma\rangle + \sum_{0\leq \beta<\alpha} \langle \dive(\p^\beta_x \sigma \p^{\alpha-\beta}_x v), \p^\alpha_x \sigma\rangle,    
    \end{align}
    so that we obtain
    \begin{align} \label{deriesti}
        &|\langle \nabla^\ell \dive(\sigma v),\nabla^\ell \sigma \rangle| \nonumber \\
        &\lesssim \|v\|_{\dot{H}^1\cap\dot{H}^{k+1}}\|\sigma\|_{\dot{H}^2\cap\dot{H}^{k}}^2+\|v\|_{\dot{H}^2\cap\dot{H}^{k+1}}\|\sigma\|_{\dot{H}^1\cap\dot{H}^{k}}\|\sigma\|_{\dot{H}^2\cap\dot{H}^{k}},
    \end{align}
    for any $\ell=1,\ldots,k$. Sobolev's inequality implies that
    \begin{align} \label{R1R2}
        &\|R_1\|_{\dot{H}^{\ell-1}}+\|R_2\|_{\dot{H}^{\ell-1}} \nonumber\\
        &\hspace{20pt}\lesssim \|(v,\eta)\|_{\dot{H}^2\cap\dot{H}^{k+1}}\|(\sigma,v,\eta)\|_{\dot{H}^1\cap\dot{H}^{k+1}} +\|f\|_{H^{k-1}}
    \end{align}
    for $\ell=1,\ldots,k$.
    For $\gamma>0$, we define the energy function
    \begin{align*}
        E_k(t):=\sum_{\ell=1}^k\left\|\nabla^\ell\left(\frac{p_1}{\rho_\infty^2}\sigma,\, v, \,\frac{c}{\theta_\infty}\eta\right)(t)\right\|_{L^2}^2 + \gamma\sum_{\ell=2}^k\langle \nabla^{\ell-1} v(t), \nabla^{\ell-1} \nabla\sigma (t)\rangle
    \end{align*}
    for $0\leq t\leq T_1$. By Lemma \ref{bilinearlemma}, we have
    \begin{align*}
        E_k  \lesssim \|(\sigma,v,\eta)\|_{\Bi{\frac{1}{2}}}^2+\|(v,\eta)\|_{\dot{H}^{2}\cap\dot{H}^{k+1}}^2 + \|\sigma\|_{\dot{H}^2\cap\dot{H}^k}^2.
    \end{align*}
    for $0\leq t\leq T_1$.
    Combining (\ref{energyid}), (\ref{energyidsigma}), (\ref{deriesti}) and (\ref{R1R2}), if $\delta_1>0$ and $\gamma>0$ are small enough, then there is a constant $c>0$ such that
    \begin{align*}
        &\frac{d}{dt}E_{k} + cE_k \lesssim  \|(\sigma,v,\eta)\|_{C_T(\Bi{\frac{1}{2}})}^2+\|f\|_{C_{T_1}(H^{k-1})}^2\ \ \ \mathrm{for}\ \ \ 0\leq t\leq T_1.
    \end{align*}
    Therefore, Gr\"{o}nwall's inequality implies
    \begin{align*}
        E_k(t) &\lesssim e^{-ct}E_k(0)+\int_0^t e^{-c\tau}d\tau  (\|(\sigma,v,\eta)\|_{C_{T_1}(\Bi{\frac{1}{2}})}^2+\|f\|_{C_{T_1}(H^{k-1})}^2)\\
        &\lesssim  E_k(0)+\|(\sigma,v,\eta)\|_{C_{T_1}(\Bi{\frac{1}{2}})}^2+\|f\|_{C_{T_1}(H^{k-1})}^2\ \ \ \mathrm{for}\ \ \ 0\leq t\leq T_1,
    \end{align*}
    so that we obtain the estimate (\ref{highesti}) when $\gamma>0$ is small.
\end{proof}
Propositions \ref{low} and \ref{high} now lead to Proposition \ref{aprioriesti}. The following proposition shows the time-weighted estimate for the $\dot{H}^4$ norm. The proof is carried out in a similar manner to that of Proposition \ref{high}.
\begin{prop} \label{highd}
    Under the assumption of Proposition $\ref{diffdecayapri}$, we have the estimate
    \begin{align}  \label{dot4desti}
        &\sup_{0\leq t\leq T_1}(1+t)^{\frac{3}{4}-\frac{\ep+s_0}{2}}\|(\tilde{\sigma},\tilde{v},\tilde{\eta})(t)\|_{\dot{H}^{4}}\nonumber\\
        &\hspace{60pt}\lesssim_{\ep}  \tilde{D}(T_1)+\delta_1\mathcal{D}_{\ep}(T_1)+ \|(\tilde{\sigma}_0,\tilde{v}_0,\tilde{\eta}_0)\|_{\Bi{s_0}\cap\dot{H}^{4}},
    \end{align}
    where 
    \begin{align*}
        \tilde{D}(T_1)=\sup_{0\leq t\leq T_1}(1+t)^{\frac{3}{4}-\frac{\ep+s_0}{2}}\|(\tilde{\sigma},\tilde{v},\tilde{\eta})(t)\|_{\Bi{\frac{3}{2}-\ep}}
    \end{align*}
\end{prop}
\begin{proof}
    Since $(\sigma_1,v_1,\eta_1)$ and $(\sigma_2,v_2,\eta_2)$ satisfy (\ref{eq}), respectively, we have the energy estimate for $(\tilde{\sigma},\tilde{v},\tilde{\eta})=(\sigma_1,v_1,\eta_1)-(\sigma_2,v_2,\eta_2)$:
    \begin{align} \label{energyid4}
        &\frac{1}{2} \p_t \left\|\nabla^4\left(\frac{p_1}{\rho_\infty^2}\tilde{\sigma},\, \tilde{v}, \,\frac{c}{\theta_\infty}\tilde{\eta}\right)\right\|_{L^2}^2 \nonumber\\
        &+ \frac{\mu}{\rho_\infty}\|\nabla^{5} \tilde{v}\|_{L^2}^2 + \frac{\mu+\mu'}{\rho_\infty}\|\nabla^{4}\dive\,\tilde{v}\|_{L^{2}}^2 + \frac{\kappa}{\rho_\infty \theta_\infty}\|\nabla^{5}\tilde{\eta}\|_{L^2}^2 \nonumber\\
        &= \frac{p_1}{\rho_\infty^2}\langle \nabla^{4}(\dive(\tilde{\sigma} v_1+\sigma_2 \tilde{v})), \nabla^{4} \tilde{\sigma}\rangle + \langle \nabla^4 \tilde{R}_1,\nabla^4 \tilde{v}\rangle \nonumber\\
        &\hspace{30pt}+\left\langle\nabla^4\left( \left(\frac{1}{\theta_1}-\frac{1}{\theta_2}\right) {R}_2\right), \nabla^{4}\tilde{\eta}\right\rangle + \left\langle \nabla^{4}\left(\frac{1}{\theta_2} \tilde{R}_2\right), \nabla^4\tilde{\eta}\right\rangle,
    \end{align}
    and
    \begin{align} \label{energyidsigma4}
        &\partial_t \langle \nabla^{3} \tilde{v}, \nabla^{3} \nabla\tilde{\sigma} \rangle +  \frac{P'(\rho_\infty)}{\rho_\infty}\|\nabla^4\tilde{\sigma}\|_{L^2}^2 \nonumber\\
        &= \rho_\infty\|\nabla^{3}\dive\,\tilde{v}\|_{L^2}^2 -\langle\nabla^{3}\tilde{v},\nabla^{3}\nabla \dive(\tilde{\sigma}v_1+\sigma_2 \tilde{v})\rangle  \nonumber\\
        &\hspace{70pt}+\left\langle \nabla^{3}\left(\frac{1}{\rho_\infty}\mathcal{A}\tilde{v}+\tilde{R_1}\right),\nabla^{3}\nabla\tilde{\sigma}\right\rangle,
    \end{align}
    where $\tilde{R}_1=R_1(\sigma_1,v_1,\eta_1)-R_1(\sigma_2,v_2,\eta_2)$, $\tilde{R}_2=R_2(\sigma_1,v_1,\eta_1)-R_2(\sigma_2,v_2,\eta_2)$. By Sobolev's inequality and Lemma \ref{bilinearlemma}, we have
    \begin{align*}
         &\langle \nabla^4 \tilde{R}_1,\nabla^4 \tilde{v}\rangle+\left\langle\nabla^4\left( \left(\frac{1}{\theta_1}-\frac{1}{\theta_2}\right) {R}_2\right), \nabla^{4}\tilde{\eta}\right\rangle + \left\langle \nabla^{4}\left(\frac{1}{\theta_2} \tilde{R}_2\right), \nabla^4\tilde{\eta}\right\rangle\\
         &\lesssim \delta_1\|(\tilde{\sigma},\tilde{v},\tilde{\eta})\|_{\Bi{\frac{3}{2}-\ep}\cap\dot{H}^4}\|\nabla^5(\tilde{v},\tilde{\eta})\|_{L^2}+ \delta_1 \|\nabla^5(\tilde{v},\tilde{\eta})\|_{L^2}^2
    \end{align*}
    and
    \begin{align*}
         \langle\nabla^{3}\tilde{v},\nabla^{3}\nabla \dive(\tilde{\sigma}v_1+\sigma_2 \tilde{v})\rangle+\left\langle \nabla^{3}\left(\frac{1}{\rho_\infty}\mathcal{A}\tilde{v}+\tilde{R_1}\right),\nabla^{3}\nabla\tilde{\sigma}\right\rangle\\
         \lesssim \delta_1 \|(\tilde{\sigma},\tilde{v})\|_{\Bi{\frac{3}{2}-\ep}\cap\dot{H}^3} \|(\nabla^5\tilde{v},\nabla^4\tilde{\sigma})\|_{L^2} + \|\nabla^5\tilde{v}\|_{L^2} \|\nabla^4\tilde{\sigma}\|_{L^2}.
    \end{align*}
    By using the identity (\ref{alphaiden}), we have
    \begin{align*}
        &\langle \nabla^{4}(\dive(\tilde{\sigma} v_1 +\sigma_2 \tilde{v})), \nabla^{4} \tilde{\sigma}\rangle \\
        &\hspace{30pt}\lesssim \delta_1 ( \|(\tilde{\sigma},\tilde{v})\|_{\Bi{\frac{3}{2}-\ep}\cap\dot{H}^4} +\|\nabla^5\tilde{v}\|_{L^2})\|\nabla^4\tilde{\sigma}\|_{L^2}.
    \end{align*}
    Therefore, there exists a constant $c>0$ such that the energy
    \begin{align*}
        E_{\gamma}(t):=  \left\|\nabla^4\left(\frac{p_1}{\rho_\infty^2}\tilde{\sigma},\, \tilde{v}, \,\frac{c}{\theta_\infty}\tilde{\eta}\right)(t)\right\|_{L^2}^2+\gamma\langle \nabla^3\tilde{v}(t),\nabla^3\nabla\tilde{\sigma}(t)\rangle,\ \ \ 0\leq t\leq T_1
    \end{align*}
    satisfies the differential inequality
    \begin{align*}
        \frac{d}{dt}E_\gamma + cE_\gamma \lesssim (\tilde{D}(T_1)+\delta_1\mathcal{D}_\ep(T_1))^2 (1+t)^{-\frac{3}{2}+{\ep+s_0}}
    \end{align*}
    when $\gamma>0$ and $\delta_1>0$ are sufficiently small. Then, Gr\"{o}nwall's inequality implies the estimate $(\ref{dot4desti})$.
\end{proof}
Now, Proposition \ref{diffdecayapri} is obtained by Proposition \ref{lowd} and Proposition \ref{highd}.
\subsection{Local existence}
We will use the following local existence result to prove Theorem \ref{timethm}. The proof is carried out by establishing the local existence and estimates for both the transport equation and the quasilinear parabolic equation, and then combining these to construct a sequence of approximate solutions converging to the local solution.
\begin{prop} \label{localexis}
    Let $k\in\mathbb{Z}_{\geq3}$ Then there exist a constant $c_2$ with $0<c_2<1$ such that if $(\sigma_0,v_0,\eta_0)\in \Bi{1/2}\cap \dot{H}^{k}$, 
    \begin{align*}
        \|\sigma_0\|_{\Bo{\frac{3}{2}}} \leq c_2 \rho_\infty
    \end{align*}
     and $f\in C_{T_2}(\Bi{-1/2}\cap\dot{H}^{k-1})$, then there exist constants $T_2>0$ and $\tilde{c}_2$ with $0<\tilde{c_2}<1$ and a unique solution $(\rho,v,\eta)=(\sigma+\rho_\infty,v,\eta+\theta_\infty)$ of $(\ref{eq})$ in $(0,T_2)\times \mathbb{R}^3$ satisfying $(\rho,v,\theta)|_{t=0}=(\rho_0,v_0,\theta_0)=(\sigma_0+\rho_\infty,v,\eta)$, $(\sigma,v,\eta) \in C_{T_2}(\Bi{1/2}\cap\dot{H}^{k})$, $(v,\eta)\in L^{2}_{T_2}(\dot{H}^{k+1})$, 
    \begin{align*}
        \|(\sigma,v,\eta)\|_{C_{T_2}(\Bi{\frac{1}{2}}\cap\dot{H}^k)}\leq 2 \|(\sigma_0,v_0,\eta_0)\|_{\Bi{\frac{1}{2}}\cap\dot{H}^k}
    \end{align*} 
    and
    \begin{align*}
        \|\sigma\|_{C_{T_2}(L^\infty)}\leq \tilde{c}_2\rho_\infty.
    \end{align*}
\end{prop}
\begin{proof}
    Lemma \ref{em}, there is a constant $C_2> 1$ such that
    \begin{align*}
        \|\tilde{\sigma}\|_{L^{\infty}}\leq C_2\|\tilde{\sigma}\|_{\Bo{\frac{3}{2}}}\ \ \ \mathrm{for\ any}\ \ \ \tilde{\sigma}\in\Bo{\frac{3}{2}}. 
    \end{align*}
    Let $c_2=1/(2C_2)$ and assume that an initial value $\sigma_0$ satisfies
    \begin{align*}
        \|\sigma_0\|_{\Bo{\frac{3}{2}}}\leq c_2 \rho_\infty.
    \end{align*} 
    Then, we have 
    \begin{align*}
        \inf_{x\in\mathbb{R}^3}\rho_0(x)\geq \frac{1}{2}\rho_\infty.
    \end{align*}
    For fixed $T_2$ with $0<T_2\leq 1$, which we will choose later, we construct the approximate sequence $\{(\rho_n,v_n,\eta_n)\}_{n\geq 0}=\{(\rho_\infty+\sigma_n,v_n,\eta_n)\}_{n\geq 0}$ with
    \begin{align*}
        \{(\sigma_n,v_n,\eta_n)\}_{n\geq 0}\subset C_{T_2}(\Bi{1/2}\cap\dot{H}^k)
    \end{align*}
    by induction.  Let $n\in\mathbb{Z}_{>0}$. Assume $(\sigma_{n-1},v_{n-1},\eta_{n-1})\in C_{T_2}(\Bi{1/2}\cap\dot{H}^{k})$ satisfies $(v_{n-1},\eta_{n-1})\in C_{T_2}(\dot{H}^{k+1})$,
    \begin{align*}
        &\|(\sigma_{n-1},v_{n-1},\eta_{n-1})\|_{C_{T_2}(\Bi{\frac{1}{2}}\cap\dot{H}^k)} + \|(v_{n-1},\eta_{n-1})\|_{L^2_{T_2}(\dot{H}^{k+1})}\\
        &\hspace{70pt}\leq2\|(\sigma_0,v_0,\eta_0)\|_{\Bi{\frac{1}{2}}\cap\dot{H}^k}=:2 E_0
    \end{align*}
    and
    \begin{align*}
        \|\sigma_{n-1}\|_{C_{T_2}(\Bi{\frac{3}{2}})} \leq \tilde{c}_2 \rho_\infty,
    \end{align*}
    where $\tilde{c}_2:=(1+c_2)/2$.
    Let $\sigma_n$ be a solution of the initial value problem of the transport equation
    \begin{equation} \label{transequation}
        \left\{
        \begin{aligned}
            &\partial_t \sigma_n +v_{n-1}\cdot \nabla \sigma_n= \tilde{F}_{n-1}\ \ \ \mathrm{in}\ \ \ (0,T_2)\times\mathbb{R}^3,\\
            &\sigma_n|_{t=0}=\sigma_0\ \ \ \mathrm{on}\ \ \ \mathbb{R}^3,
        \end{aligned}
        \right.
    \end{equation}
    where
    \begin{align*}
        \tilde{F}_{n-1}= -\sigma_{n-1} \dive\, v_{n-1}.
    \end{align*}
    By the identities
    \begin{align*}
        -\langle v_{n-1}\cdot\nabla \nabla^k \sigma_n, \nabla^k \sigma_n\rangle = \frac{1}{2} \langle \dive\,v_{n-1} \nabla^k\sigma_n,\nabla^k\sigma_n\rangle
    \end{align*}
    and, for any $j\in\mathbb{Z}$,
    \begin{align*}
        &-2^{j}\langle \delj(v_{n-1}\cdot\nabla\sigma_n),\delj\sigma_n\rangle \\
        &\hspace{50pt}= \frac{1}{2}2^j\langle \dive\,v_{n-1} \delj\sigma_n,\delj\sigma_n\rangle - 2^j\langle [\delj,v_{n-1}\cdot\nabla] \sigma_n,\delj \sigma_n\rangle,
    \end{align*}
    it follows from Lemma \ref{bilinearlemma} that, for any $0\leq t\leq T_2$ and $\alpha\in\mathbb{Z}^3_{\geq0}$, 
    \begin{align*}
        &\|\p^{\alpha}_x \sigma_n(t)\|_{L^2}^2- \|\p^{\alpha}_x \sigma_0\|_{L^2}^2 = \int_0^t \langle \dive\,v_{n-1}\p^{\alpha}_x\sigma_n +2\p^{\alpha}_x \tilde{F}_{n-1},\p^{\alpha}_x \sigma_n\rangle d\tau\\
        &\hspace{100pt} + \sum_{0<\beta\leq \alpha}\int_0^t \langle \p_x^\beta v_{n-1}\cdot\p_x^{\alpha-\beta}\nabla \sigma_n,\p_x^{\alpha}\sigma_n \rangle  d\tau \\
        &\leq T_2 \|\dive\, v_{n-1}\|_{C_{T_2}(\dot{H}^1\cap\dot{H}^{k})}\|\sigma_n\|^2_{C_{T_2}(\dot{H}^{1}\cap\dot{H}^k)}+T_2 E_0^2\|\p^{\alpha}_x\sigma_n\|_{C_{T_2}(L^2)},
    \end{align*}
    \begin{align*}
       &\|\sigma_n(t)\|_{\Bo{\frac{3}{2}}}^2 - \|\sigma_0\|_{\Bo{\frac{3}{2}}}^2 = \sum_{j\in\mathbb{Z}}2^{3j}\int_0^t \langle \delj (v_{n-1}\cdot\nabla\sigma_n +2 \tilde{F}_{n-1}),\delj \sigma_n\rangle d\tau\\
       &\hspace{5pt}\lesssim    \|\nabla v_{n-1}\|_{L^1_{T_2}(\Bo{\frac{3}{2}})} \|\sigma_n\|_{C_{T_2}(\Bo{\frac{3}{2}})}^2 \\
       &\hspace{20pt}+ (T_2 E_0 +T_2^{\frac{1}{2}}\|v_{n-1}\|_{L^2_{T_2}(\dot{H}^{k+1})})\|\sigma_{n-1}\|_{C_{T_2}(\Bo{\frac{3}{2}})} \|\sigma_n\|_{C_{T_2}(\Bo{\frac{3}{2}})}\\
       &\hspace{5pt}\lesssim T_2^{\frac{1}{2}} E_0 \|\sigma_n\|_{C_{T_2}(\Bo{\frac{3}{2}})}^2 + T_2^{\frac{1}{2}} E_0\|\sigma_{n-1}\|_{C_{T_2}(\Bo{\frac{3}{2}})} \|\sigma_n\|_{C_{T_2}(\Bo{\frac{3}{2}})}
    \end{align*}
    and 
    \begin{align*}
       \|\sigma_n\|_{C_{T_2}(\Bi{\frac{1}{2}})}^2 - \|\sigma_0\|_{\Bi{\frac{1}{2}}}^2\lesssim   T_2^{\frac{1}{2}} E_0\|\sigma_n\|_{C_{T_2}(\Bi{\frac{1}{2}})}^2 +T_2 E_0^2\|\sigma_n\|_{C_{T_2}(\Bi{\frac{1}{2}})},
    \end{align*}
    so that there exist constants $\tilde{c}>0$ such that if
    \begin{align*}
        T_2 (E_0+E_0^2) \leq \tilde{c},
    \end{align*}
    then we obtain
    \begin{align*}
        \|\sigma_n\|_{C_{T_2}(\Bi{\frac{1}{2}}\cap \dot{H}^k)}\leq 2E_0
    \end{align*}
    and 
    \begin{align} \label{rholowb}
        \|\sigma_n\|_{C_{T_2}(\Bo{\frac{3}{2}})} \leq \tilde{c_2} \rho_\infty.
    \end{align}
    Let $(v_n,\eta_n)$ be a solution of the initial value problem
    \begin{equation} \label{vetaeq}
        \left\{
        \begin{aligned}
            &\partial_t v_n   - \frac{1}{\rho_{n-1}}\mathcal{A} v_n=\tilde{G}_{n-1}\ \ \ \mathrm{in}\ \ \ (0,T_2)\times\mathbb{R}^3,\\
            &\partial_t \eta_n -\frac{\kappa}{c\rho_\infty}\Delta \eta_n = \tilde{H}_{n-1}\ \ \ \mathrm{in}\ \ \ (0,T_2)\times\mathbb{R}^3,\\
            &(v_n,\eta_n)|_{t=0}=(v_0,\eta_0)\ \ \ \mathrm{on}\ \ \ \mathbb{R}^3
        \end{aligned}
        \right.
    \end{equation}
    where 
    \begin{align*}
        \tilde{G}_{n-1}=-v_{n-1}\cdot\nabla v_{n-1} -\frac{\nabla (P(\rho_\infty+\sigma_{n-1},\theta_\infty+\eta_{n-1}))}{\rho_\infty+\sigma_{n-1}} + f
    \end{align*}
    and
    \begin{align*}
        &\tilde{H}_{n-1} = -v_{n-1}\cdot\nabla\eta_{n-1} \\
        &\hspace{20pt}- \frac{(\theta_\infty+\eta_{n-1})\p_\theta P(\rho_\infty+\sigma_{n-1},\theta_\infty+\eta_{n-1})}{c(\rho_\infty+\sigma_{n-1})}\dive v_{n-1} + \frac{\Psi(v_{n-1})}{1+\sigma_{n-1}}.
    \end{align*}
    Then, for any $0\leq t\leq T_2$, 
    \begin{align*}
        &\|\nabla^k v_n(t)\|_{L^2}^2 - \|\nabla^k v_0\|_{L^2}^2\\
        &\hspace{50pt}+\int_0^t  \frac{\mu}{\rho_\infty+\sigma_{n-1}}\|\nabla^{k+1}v_n\|_{L^2}^2+ \frac{\mu+\mu'}{\rho_\infty+\sigma_{n-1}}\|\nabla^{k+1}v_n\|_{L^2}^2 d\tau\\
        &\hspace{10pt}\lesssim\int_0^t( \|\sigma_{n-1}\|_{\dot{H}^1\cap\dot{H}^k}\|v_n\|_{\dot{H}^1\cap\dot{H}^k}+\|\tilde{G}_{n-1}\|_{\dot{H}^{k-1}})\|\nabla^{k+1} v_n\|_{L^2}  d\tau,
    \end{align*}
    \begin{align*}
        &\|\nabla^k \eta_n(t)\|_{L^2}^2 - \|\nabla^k \eta_0\|_{L^2}^2+\int_0^t  \frac{\kappa}{c(\rho_\infty+\sigma_{n-1})}\|\nabla^{k+1}\eta_n\|_{L^2}^2 d\tau\\
        &\hspace{10pt}\lesssim\int_0^t( \|\sigma_{n-1}\|_{\dot{H}^1\cap\dot{H}^k}\|\eta_n\|_{\dot{H}^1\cap\dot{H}^k}+\|\tilde{H}_{n-1}\|_{\dot{H}^{k-1}})\|\nabla^{k+1} \eta_n\|_{L^2}  d\tau.
    \end{align*}
    For any $j\in\mathbb{Z}$ and $t$ with $0\leq t\leq T_2$, 
    \begin{align*}
        &2^j\|\delj (v_n,\eta_n)(t)\|_{L^2}^2 - 2^j\|\delj (v_0,\eta_0)\|_{L^2}^2\\
        &\hspace{30pt}+2^j\int_0^t  \mu\|\nabla\delj v_n\|_{L^2}^2+(\mu+\mu') \|\dive\delj v_n\|_{L^2}^2+\frac{k}{c}\|\nabla \delj\eta_n\|_{L^2}^2 d\tau\\
        &\hspace{10pt}\lesssim\int_0^t\|\sigma_{n-1}\|_{\Bi{\frac{1}{2}}}\|(v_n,\eta_n)\|_{\Bi{\frac{1}{2}}\cap\dot{H}^{k+1}}2^{\frac{1}{2}j}\|\nabla\delj(v_n,\eta_n)\|_{L^2} d\tau\\
        &\hspace{130pt}+\int_0^t \|(\tilde{G}_{n-1},\tilde{H}_{n-1})\|_{L^2}2^{j}\|\delj(v_n,\eta_n)\|_{L^2}  d\tau.
    \end{align*}
    Since
    \begin{align*}
        \|(\tilde{G}_{n-1},\tilde{H}_{n-1})\|_{{H}^{k-1}}\lesssim E_0 + E_0^2,
    \end{align*}
    there is a constant $C>0$ such that
    \begin{align*}
        &\|(v_n,\eta_n)\|_{C_{T_2}(\Bi{\frac{1}{2}}\cap\dot{H}^k)}+ \|(v_n,\eta_n)\|_{L^2_{T_2}(\dot{H}^{k+1})}\\
        &\hspace{50pt}\leq \|(\sigma_0,v_0,\eta_0)\|_{\Bi{\frac{1}{2}}\cap\dot{H}^k}+ T_0^{\frac{1}{4}}(E_0 + E_0^2).
    \end{align*}
    Therefore, if we take $\tilde{c}=\tilde{c}(E_0)>0$ small enough, then we have 
    \begin{align} \label{nn}
        \|(\sigma_n,v_n,\eta_n)\|_{C_{T_2}(\Bi{\frac{1}{2}}\cap\dot{H}^k)}+ \|(v_n,\eta_n)\|_{L^2_{T_2}(\dot{H}^{k+1})} \leq 2E_0
    \end{align}
    when $T_2^{\frac{1}{4}}(E_0+E_0^2)\leq \tilde{c}$.  Now, we obtain the sequence $\{(\rho_n,v_n,\eta_n)\}_{n\geq 0}=\{(\rho_\infty+\sigma_n,v_n,\eta_n)\}_{n\geq 0}$
    \begin{align*}
        \{(\sigma_n,v_n,\eta_n)\}_{n\geq 0}\subset C_{T_2}(\Bi{1/2}\cap\dot{H}^k)
    \end{align*}
    for fixed $T_2>0$ with $T_2^{\frac{1}{4}}(E_0+E_0^2)\leq \tilde{c}$, which satisfies (\ref{transequation}), (\ref{rholowb}), (\ref{vetaeq}) and (\ref{nn}) for any $n\in\mathbb{Z}_{> 0}$. Let $\delta\sigma_n=\sigma_n-\sigma_{n-1}$, $\delta v_n=v_n-v_{n-1}$ and $\delta\eta_n=\eta_{n}-\eta_{n-1}$ for $n\in\mathbb{Z}_{>0}$. We next show the contractivity estimate
    \begin{align} \label{contracesti}
        \|(\delta\sigma_n,\delta v_n,\delta\eta_n)\|_{C_{T_2}(\Bi{\frac{1}{2}})} 
        \leq \frac{1}{2} \|(\delta\sigma_{n-1},\delta v_{n-1},\delta\eta_{n-1})\|_{C_{T_2}(\Bi{\frac{1}{2}})},\ \ n\in\mathbb{Z}_{\geq 2}
    \end{align}
    for small $T_2>0$.
    The pair of functions $(\delta v_n,\delta \eta_n)$ is a solution of the initial value problem
    \begin{equation} \label{delvetaeq}
        \left\{
        \begin{aligned}
            &\partial_t \delta v_n   - \frac{1}{\rho_{n-1}}\mathcal{A} \delta v_n \\
            &\hspace{30pt}=-\frac{\delta\sigma_{n-1}}{\rho_{n-1}\rho_{n-2}}\mathcal{A}v_{n-1}+\delta\tilde{G}_{n-1}\ \ \ \mathrm{in}\ \ \ (0,T_2)\times\mathbb{R}^3,\\
            &\partial_t \delta\eta_n -\frac{\kappa}{c\rho_{n-1}}\Delta\delta \eta_n = \delta\tilde{H}_{n-1}\ \ \ \mathrm{in}\ \ \ (0,T_2)\times\mathbb{R}^3,\\
            &(\delta v_n,\delta \eta_n)|_{t=0}=(0,0)\ \ \ \mathrm{on}\ \ \ \mathbb{R}^3,
        \end{aligned}
        \right.
    \end{equation}
    where $\delta G_{n-1}=\tilde{G}_{n-1}-\tilde{G}_{n-2}$, $\delta \tilde{H}_{n-1}=\tilde{H}_{n-1}-\tilde{H}_{n-2}$.
    By Lemma \ref{commu}, for any $j\in\mathbb{Z}$,
    \begin{align*}
        &2^{-\frac{1}{2}j}\left\|\left[\delj,\rho_{n-1}^{-1}\dive\right]\nabla \delta v_n + \left[\delj,\rho_{n-1}^{-1}\nabla\right]\dive \delta v_n \right\|_{L^2}  \\
        &\lesssim \|\nabla\sigma_{n-1}\|_{\Bo{\frac{3}{2}}}\|\delta v_n\|_{\Bi{\frac{1}{2}}}.
    \end{align*}
    By (\ref{rholowb}), for any $j\in\mathbb{Z}$,
    \begin{align*}
        &2^j \left\langle\frac{1}{\rho_{n-1}}\mathcal{A}\delj\delta v_n, \delj\delta v_n\right\rangle \\
        &\hspace{40pt}\gtrsim 2^j\|\nabla\delj\delta v_n\|_{L^2}^2 -  \|\nabla \sigma_{n-1}\|_{L^\infty}^2 \|\delta v_n\|_{\Bi{\frac{1}{2}}}^2.
    \end{align*}
    By Lemma \ref{bilinearlemma}, we have
    \begin{align*}
        &\left\|\frac{\delta\sigma_{n-1}}{\rho_{n-1}\rho_{n-2}}\mathcal{A}v_{n-1}\right\|_{\Bi{-\frac{1}{2}}} \lesssim \|\delta\sigma_{n-1}\|_{\Bi{\frac{1}{2}}}\|v_n\|_{\Bi{\frac{5}{2}}}.
    \end{align*}
    Then, for any $j\in\mathbb{Z}$ and $0\leq t\leq T_2$, we have
    \begin{align*}
        &2^j\|\delj\delta v_n(t)\|_{L^2}^2 + \int_0^{t} 2^j\|\nabla\delj \delta v_n\|_{L^2}^2 d\tau\\
        &\lesssim_{E_0} \int_0^t \|\nabla\sigma_{n-1}\|_{\Bo{\frac{3}{2}}}^2 \|\delta v_n\|_{\Bi{\frac{1}{2}}}^2 + \left\|\frac{\delta\sigma_{n-1}}{\rho_{n-1}\rho_{n-2}}\mathcal{A}v_{n-1}\right\|_{\Bi{-\frac{1}{2}}}^2 \\
        &\hspace{150pt} +\|\delta \tilde{G}_{n-1}\|_{\Bi{-\frac{1}{2}}}^2 d\tau\\
        &\lesssim T_2 E_0^2 \|(\delta v_{n},\delta \sigma_{n-1})\|_{C_{T_2}(\Bi{\frac{1}{2}})}^2 + T_2\|\delta \tilde{G}_{n-1}\|_{C_{T_2}(\Bi{-\frac{1}{2}})}^2.
    \end{align*}
    Since $\delta \eta_n$ is a solution of (\ref{delvetaeq}), for any $j\in\mathbb{Z}$ and $0\leq t\leq T_2$, we have
    \begin{align*}
        2^j \|\delj \delta \eta_n(t)\|_{L^2}^2 +\int_0^t 2^j \|\nabla\delj\delta\eta_n\|_{L^2}^2 d\tau\lesssim T_2 \|\delta \tilde{H}_{n-1}\|_{C_{T_2}(\Bi{-\frac{1}{2}})}^2.
    \end{align*}
    Thus, we obtain
    \begin{align} \label{netuesti}
        &\|(\delta v_n,\delta\eta_n)\|_{C_{T_2}(\Bi{\frac{1}{2}})} \lesssim T_2^{\frac{1}{2}}E_0\|(\delta v_n,\delta\sigma_{n-1})\|_{C_{T_2}(\Bi{\frac{1}{2}})}  \nonumber\\
        &\hspace{100pt}+ T_2^{\frac{1}{2}} \|(\delta \tilde{G}_{n-1},\delta\tilde{H}_{n-1})\|_{C_{T_2}(\Bi{-\frac{1}{2}})}.
    \end{align}
    The function $\delta \sigma_n$ satisfies the following initial value problem:
    \begin{equation} \label{deltatranseq}
        \left\{
        \begin{aligned}
            &\partial_t \delta\sigma_n +v_{n-1}\cdot \nabla \delta\sigma_n=-\delta v_{n-1}\cdot\nabla\sigma_{n-1}+ \delta\tilde{F}_{n-1}\ \ \mathrm{in}\ \ (0,T_2)\times\mathbb{R}^3,\\
            &\delta\sigma_n|_{t=0}=0\ \ \ \mathrm{on}\ \ \ \mathbb{R}^3,
        \end{aligned}
        \right.
    \end{equation}
    where $\delta \tilde{F}_{n-1}=\tilde{F}_{n-1}-\tilde{F}_{n-2}$. 
    By Lemma \ref{commu} and Lemma \ref{bilinearlemma}, we obtain
    \begin{align} \label{aaaa}
        &2^j\|\delj \delta\sigma_n(t)\|_{L^2}^2 + 2^j\int_0^t \langle v_{n-1}\cdot\nabla \delj \delta \sigma_n, \delj\delta\sigma_n\rangle d\tau \nonumber\\
        &\lesssim \int_0^t  2^j\|[\delj,v_{n-1}\cdot\nabla]\delta\sigma_n\|_{L^2}^2 \nonumber\\
        &\hspace{50pt}+ \|-\delta v_{n-1}\cdot\nabla\sigma_{n-1}+\delta \tilde{F}_{n-1}\|_{\Bi{\frac{1}{2}}}^2 d\tau\nonumber\\
        &\lesssim T_2E_0^2\|(\delta \sigma_n,\delta v_{n-1})\|_{C_{T_2}(\Bi{\frac{1}{2}})}^2+T_2\|\delta\tilde{F}_{n-1}\|_{C_{T_2}(\Bi{\frac{1}{2}})}^2.
    \end{align}
    By the estimate
    \begin{align*}
        \langle v_{n-1}\cdot\nabla\delj\delta\sigma_n, \delj \delta \sigma_n\rangle &= -\frac{1}{2}\langle \dive v_{n-1}, (\delj\delta\sigma_n)^2\rangle \\
        &\lesssim \|\dive v_{n-1}\|_{L^\infty}\|\delj\delta\sigma_{n}\|_{L^2}^2,
    \end{align*}
    it follows from (\ref{aaaa}) that
    \begin{align} \label{deltransesti}
        &\|\delta\sigma_{n}\|_{C_{T_2}(\Bi{\frac{1}{2}})} \lesssim T_2^\frac{1}{2}E_0\|(\delta\sigma_n,\delta v_{n-1})\|_{C_{T_2}(\Bi{\frac{1}{2}})} \nonumber\\
        &\hspace{120pt}+ T_2^\frac{1}{2} \|\delta\tilde{F}_{n-1}\|_{C_{T_2}(\Bi{\frac{1}{2}})}.
    \end{align}
    By Lemma \ref{bilinearlemma}, we have
    \begin{align} \label{exesti}
        &\|\delta\tilde{F}_{n-1}\|_{C_{T_2}(\Bi{\frac{1}{2}})} + \|(\delta\tilde{G}_{n-1},\delta\tilde{H}_{n-1})\|_{C_{T_2}(\Bi{-\frac{1}{2}})}\nonumber\\
        &\hspace{20pt}\lesssim (E_0+E_0^2)\|(\delta\sigma_{n-1},\delta v_{n-1},\delta\eta_{n-1})\|_{C_{T_2}(\Bi{\frac{1}{2}})}.
    \end{align}
    By (\ref{netuesti}), (\ref{deltransesti}) and (\ref{exesti}), if we take sufficiently small $T_2>0$, then the sequence $\{(\sigma_n,v_n,\eta_n)\}_{n\geq 0}$ satisfies (\ref{contracesti}). Therefore, the sequence $\{(\sigma_n,v_n,\eta_n)\}_{n\geq 0}$ is a Cauchy sequence in $C_{T_2}(\Bi{1/2})$, is a bounded sequence in $C_{T_2}(\dot{H}^1\cap\dot{H}^k)$ and satisfies (\ref{rholowb}), so that the convergence limit $(\rho,v,\theta)=(\rho_\infty+\sigma,v,\theta_\infty+\eta)$ of the sequence $\{(\rho_n,v_n,\eta_n)\}$ is a solution of (\ref{eq}) which satisfies
    \begin{align*}
        \|(\sigma,v,\eta)\|_{C_{T_2}(\Bi{\frac{1}{2}})} \leq 2E_0.
    \end{align*}
    and 
    \begin{align*}
        \|\sigma\|_{C_{T_2}(L^\infty)} \leq \tilde{c}_{2}\rho_\infty.
    \end{align*}
\end{proof}

\subsection{Proof of Theorem \ref{timethm}}
We now turn to prove Theorem \ref{timethm}. The following construction of time-periodic solutions is inspired by Valli \cite{MR753158}, which concerns a time-periodic problem for the isentropic compressible Navier-Stokes equation in a bounded domain.
\begin{proof}[Proof of Theorem $\ref{timethm}$]
    By combining Proposition \ref{aprioriesti} and Proposition \ref{localexis}, there is small $\delta>0$ such that if initial data $(\rho_0,v_0,\theta_0)=(\rho_\infty,0,\theta_\infty)+U_0$ and $f$ satisfy
    \begin{align*}
        \|U_0\|_{\Bi{\frac{1}{2}}\cap\dot{H}^k}+\|f\|_{C(\mathbb{R};\Bi{\frac{1}{2}}\cap\dot{H}^{k-1})}\leq \delta,
    \end{align*}
    then there exists a global solution $(\rho,v,\theta)=(\rho_\infty,0,\theta_\infty)+U$ of $(\ref{eq})$ satisfying
    \begin{align*}
        \|U\|_{C(\Bi{\frac{1}{2}}\cap\dot{H}^k)}\lesssim \delta.
    \end{align*}
    Let $U^*$ be the global solution to (\ref{eq}) with zero initial data. Then, by taking $\delta>0$ sufficiently small, Proposition \ref{aprioriesti} and Proposition \ref{diffdecayapri} imply that
    \begin{align*}
        \|U^*\|_{C(\Bi{\frac{1}{2}}\cap\dot{H}^k)}\lesssim \delta
    \end{align*}
    and
    \begin{align*}
        \|U^*(nT)-U^*(mT)\|_{\Bi{1}\cap\dot{H}^4} \lesssim ((n-m)T)^{-\frac{1}{4}}\delta
    \end{align*}
    for any $n,m\in\mathbb{Z}_{\geq 0}$ with $ m\leq n$. Since the sequence $\{U^*(nT)\}_{n\in\mathbb{Z}_{\geq 0}}$ is bounded in $\Bi{1/2}\cap\dot{H^k}$ and is a Cauchy sequence in $\Bi{1}\cap \dot{H}^4$, there is $U^*_\infty\in \Bi{1/2}\cap\dot{H}^k$ such that
    \begin{align*}
        \lim_{n\to\infty}U^*(nT)= U^*_\infty\ \ \ \mathrm{in}\ \ \ \Bi{1}\cap \dot{H}^4
    \end{align*}
    and
    \begin{align*}
        \|U^*_\infty\|_{\Bi{\frac{1}{2}}\cap\dot{H}^k}\lesssim \delta.
    \end{align*}
    Let $(\rho_T,v_T,\theta_T)=(\rho_\infty,0,\theta_\infty)+U_T$ be a global solution of (\ref{eq}) with initial data $(\rho_\infty,0,\theta_\infty)+U^*_\infty$. 
    Then, by Proposition \ref{diffdecayapri}, for any $n\in\mathbb{Z}_{>0}$, we have
    \begin{align} \label{uniquen}
        \|U_T(T)-U^*(nT))\|_{\Bi{1}}\lesssim  \|U^*_\infty-U^*((n-1)T)\|_{\Bi{1}\cap\dot{H^4}}.
    \end{align}
    Taking limit as $n\to\infty$ on both sides of (\ref{uniquen}) yields $U_T(T)=U^*_\infty$. By the local in time uniqueness in Proposition \ref{localexis}, we conclude that $(\rho_T,v_T,\theta_T)$ is the desired time-periodic solution. The decay estimate (\ref{maindecay}) in Theorem \ref{timethm} follows immediately from Lemma \ref{fund2} (iii), Lemma \ref{em} and Proposition \ref{diffdecayapri}.
\end{proof}

\section*{Acknowledgment}
    I would like to thank Professor Yoshiyuki Kagei for his valuable comments. This work was supported by JSPS KAKENHI Grant Number JP23KJ0942.

\bibliographystyle{plain}
\bibliography{myrefs}
\end{document}